\documentclass[12pt,oneside,british,a4paper]{amsart}
	\usepackage[
  margin=2cm,
  includefoot,
  footskip=1in,
]{geometry}

\usepackage{amsmath}
\usepackage{amssymb}
\usepackage{amsfonts}

\usepackage{amstext} % for \text macro
\usepackage{array}   % for \newcolumntype macro
\newcolumntype{L}{>{$}l<{$}} % math-mode version of "l" column type

\usepackage{multicol}
\usepackage{lipsum}
\usepackage{tikz}
\usepackage{float}

\usepackage[normalem]{ulem} %to strike out or cross out text

\usepackage{IEEEtrantools}

\usepackage{amsrefs}

\usepackage{hyperref}
\usepackage{amsthm}
\usepackage{enumerate}
\usepackage[mathscr]{eucal}
\usepackage[all]{xy}
\usepackage{graphicx}
\usepackage{caption}
\usepackage{colortbl}

\usepackage{tikz}
\usepackage{bbm}

\usepackage{verbatim}

\definecolor{blue}{rgb}{0,0,1}

\definecolor{red}{rgb}{1,0,.2}

\usepackage{ulem} %allows \sout and \xout, note that it also adds underling to the bib (unless you add the code \normalulem before the bib, and it changes emph to underline throughout

\newcommand{\abs}[1]{\left \vert#1\right \vert}

\newtheorem{theorem}{Theorem}[section]
\newtheorem{lemma}[theorem]{Lemma}
\newtheorem{corollary}[theorem]{Corollary}
\newtheorem{proposition}[theorem]{Proposition}
\newtheorem{definition}[theorem]{Definition}

\numberwithin{equation}{section}
\theoremstyle{definition}

\theoremstyle{definition}

\numberwithin{equation}{section}
\theoremstyle{plain}

\theoremstyle{remark}
\newtheorem{remark}[theorem]{Remark}

\newtheorem{case[theorem]}{Case}

\usepackage{stackengine}
\newcommand{\ubar}[1]{\stackunder[1.2pt]{$#1$}{\rule{.9ex}{.075ex}}}

\newcommand{\ox}{{\bar x}}
\newcommand{\oy}{{\bar y}}
\newcommand{\oz}{{\bar z}}
\newcommand{\ux}{{\ubar x}}
\newcommand{\uy}{{\ubar y}}
\newcommand{\uz}{{\ubar z}}

\newcommand{\vs}{\vskip.125in}

\usepackage{hyperref}

\newcommand{\N}{\ensuremath{\mathbb{N}}}
\newcommand{\Q}{\ensuremath{\mathbb{Q}}}
\newcommand{\R}{\ensuremath{\mathbb{R}}}
\newcommand{\Z}{\ensuremath{\mathbb{Z}}}

\newcommand{\al}{\alpha}

\newcommand{\de}{\delta}

\usepackage{xifthen,ifthen}
\makeatletter
\newcounter{@ToDo}
\newcommand{\todo@helper}[1]{%
	({\color{red}TODO~\arabic{@ToDo}: {#1\@addpunct{.}}})%
}
\newcommand{\todo}[1]{\stepcounter{@ToDo}%
	\relax\ifmmode\text{\todo@helper{#1}}%
	\else\todo@helper{#1}\fi%
}
\author{Rajula Srivastava and Krystal Taylor}
\address{Mathematical Institute of the University of Bonn and Max Planck Institute for Mathematics, Bonn}
\address{Endenicher Allee 60, 53115, Bonn, Germany.}
\email{rajulas@math.uni-bonn.de}
\address{Department of Mathematics, The Ohio State University}
\email{taylor.2952@osu.edu}
\title[Lattice Points near Kor\'anyi spheres]{\centering{Counting Lattice Points near Kor\'anyi Spheres via Generalized Radon Transforms}}
\begin{document}

\begin{abstract} 
In this note, we study a lattice point counting problem for spheres in Heisenberg groups, incorporating \textit{both} the non-isotropic dilation structure and the non-commutative group law. More specifically, we establish an upper bound for the average number of lattice points in a $\delta$-neighborhood of a Kor\'anyi sphere of large radius, where the average considered is over \textit{Heisenberg group translations} of the sphere. This is in contrast with previous works, which either count lattice points on dilates of a \textit{fixed} sphere (see \cites{GNT, Gath2}) or consider averages over \textit{Euclidean} translations of the sphere (see \cites{CT}). We observe that incorporating the Heisenberg group structure allows us to circumvent the degeneracy arising from the vanishing of the Gaussian curvature at the poles of the Kor\'anyi sphere. In fact, in lower dimensions (the first and second Heisenberg group), our method establishes an upper bound for this number which gives a logarithmic improvement over the bound implied by the previously known results. Even for the higher dimensional Heisenberg groups, we recover the bounds implied by the main result of \cite{GNT} using a completely different approach of generalized Radon transforms.

Further, we obtain upper bounds for the average number of lattice points near more general spheres described with respect to radial, Heisenberg homogeneous norms as considered in \cite{GNT}.

\end{abstract}

\maketitle
\setcounter{tocdepth}{2}
\tableofcontents

\section{Introduction} 
In this paper, we establish upper bounds on the \textit{average} number of lattice points near Kor\'anyi spheres in the Heisenberg group. Here the average is considered over translations of the sphere with respect to the Heisenberg group law. Unlike the previous works in the area (see \cite{GNT, Gath2, CT}), we incorporate both the translation and scaling structure of the Heisenberg group into the lattice point counting problem. The ``smoothening effect" of the former is reflected in the Sobolev mapping properties of the Radon transform associated to averaging over the Kor\'anyi sphere. While the role of generalized Radon transforms in variable coefficient lattice point counting problems has been studied in \cite{IT}, it does not consider the Heisenberg group setting. As such, a primary objective of this manuscript is to bridge this gap between techniques originating from the theory of oscillatory integral operators (from Euclidean harmonic analysis) and concrete counting problems on Heisenberg groups.

More precisely, we shall use the fact that that the generalized Radon transform associated to averaging over the Kor\'anyi sphere has non-vanishing rotational curvature, in the sense that the associated Monge-Ampere determinant is non-zero (see \S\ref{sssec IT} and \eqref{eq MongeAmp} for a precise definition). This has been established previously by the first author in \cite{sri24}, in connection with Lebesgue space estimates for the Kor\'anyi maximal function (also see \cite{OSch98}). In this paper, we present a more streamlined argument adapted to the counting problem at hand. Moreover, for more general Heisenberg spheres, we show that the rank of the Monge-Ampere matrix can drop at most by one. This fact has consequences for the $L^2$ -Sobolev mapping property of the associated averaging operator. 

With these mapping properties in tow, we are able to circumvent the issue of vanishing curvature at the poles of the Kor\'anyi sphere, which the previous papers \cite{GNT, Gath2, CT, Campo} had to spend considerable effort to mitigate. We can thus reduce matters to applying energy estimates of Iosevich and the second author from \cite{IT} (for the Kor\'anyi sphere) and of 
Campolongo, and of Campolongo and the second author from \cite{Campo, CT} (for Heisenberg spheres in general) to provide a short argument for the \textit{average} number of lattice points in a small neighborhood of the spheres. 

In fact, for the Kor\'anyi spheres in lower dimensions (the first and second Heisenberg group), our method establishes an upper bound for the average number of lattice points which gives a logarithmic improvement over the bound implied by the previously known results. Even for the higher dimensional Heisenberg groups, we recover the bounds implied by the main result of \cite{GNT}. To state our results more precisely and to summarize relevant background results, we begin by introducing some notation and definitions.
\vs

Let $\mathsf{H}^n=\mathbb{R}^{2n} \times \mathbb{R}$ be the Heisenberg group of real Euclidean dimension $D=2n+1$. We shall use the notation $x=(\ux, \ox)$, $y=(\uy, \oy)$, with $\ux, \uy\in \mathbb{R}^{2n}$ and $\ox, \oy\in\mathbb{R}$, to denote elements of $\mathsf{H}^n$. The group law is given by 
\begin{equation}
\label{grouplaw}
x * y =(\ux+\uy, \ox+\oy + \tfrac{1}{2}\ux^\intercal J\uy), \end{equation}
where \[J:=\begin{pmatrix}
0 & & I_n\\
-I_n &  & 0
\end{pmatrix}\]
is the $2n\times 2n$ standard symplectic matrix (and $I_n$ is the $n\times n$ identity matrix). One may verify that $J^2 = -I_{2n}$, so that $J$ is invertible. 
For example, when $n=1$, the group law is explicitly given by
\begin{equation}\label{d=3}
x*y = \left(\underline{x} + \underline{y}, x_3 + y_3 + \frac{1}{2}(x_1 y_2 - x_2 y_1)\right). 
\end{equation}
\vs 

Further, the Kor\'anyi  norm (also called the Cygan- Kor\'anyi norm, see \cites{Cy, Kor}) of an element $x\in \mathsf{H}^n$ is defined to be \[\Vert x \Vert_4:=(|\ux|^4+16|\ox|^2)^{\frac{1}{4}},\] 
where $|\cdot|$ denotes the Euclidean norm. This norm is homogeneous of degree one with respect to the natural parabolic dilation structure $\delta_t((\ux,\ox)):=(t\ux,t^2 \ox)$ on $\mathsf{H}^n$.

Define the Kor\'anyi ball \textit{centered at the origin} and of radius $R$ to be 
\begin{equation}\label{Hballzero}
B_R^{4}(\vec{0}) := \big\{y=(\uy,\oy) \in \mathbb{R}^{2n} \textsf{ x } \mathbb{R}: |\uy|^4 + 16|\oy|^{2} \leq R^4\big\}.
\end{equation}

There exists a unique Radon measure $\mu$ on $B_R^{4}(\vec{0})$ induced by the Haar measure on $\mathsf{H}^n$. 
These balls present an interesting evolution in lattice point counting problems due to points of vanishing curvature on their surfaces.  
%Depending on which value of $\alpha$ is considered, 
The Gaussian curvature of the boundary of $B_R^{4}(\vec{0})$
vanishes at the north and south poles. Indeed, all principle curvatures of the Kor\'anyi sphere vanish at these points.

\begin{multicols}{3}
\begin{figure}[H]
  \centering
\includegraphics[scale=.25]{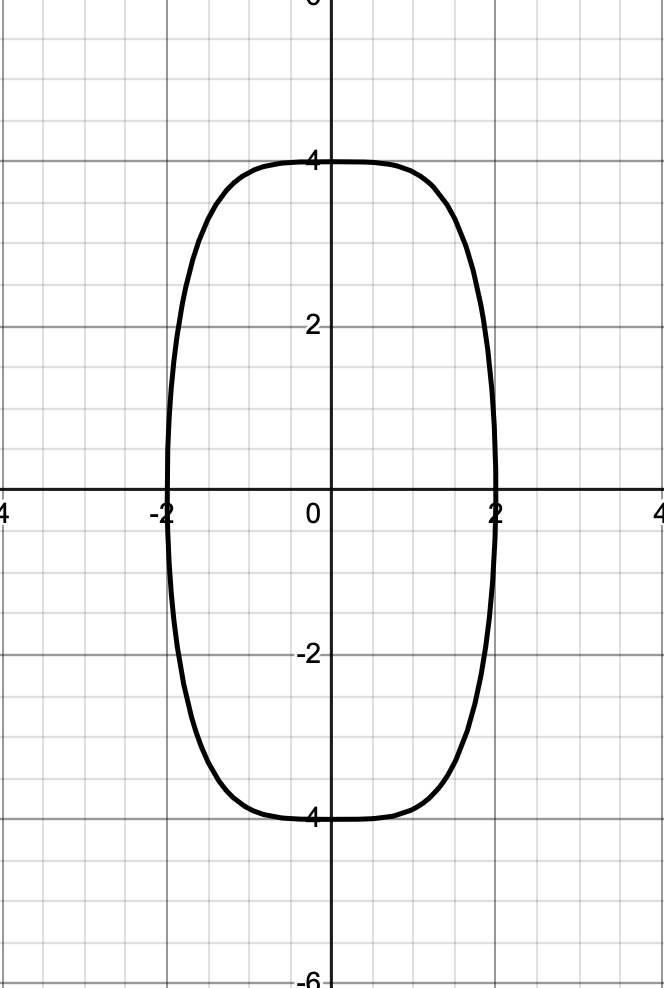}
  \caption{$x=(0,0,0)$, $y_2=x_2$ }
\end{figure}
\begin{figure}[H]
  \centering
\includegraphics[scale=.25]{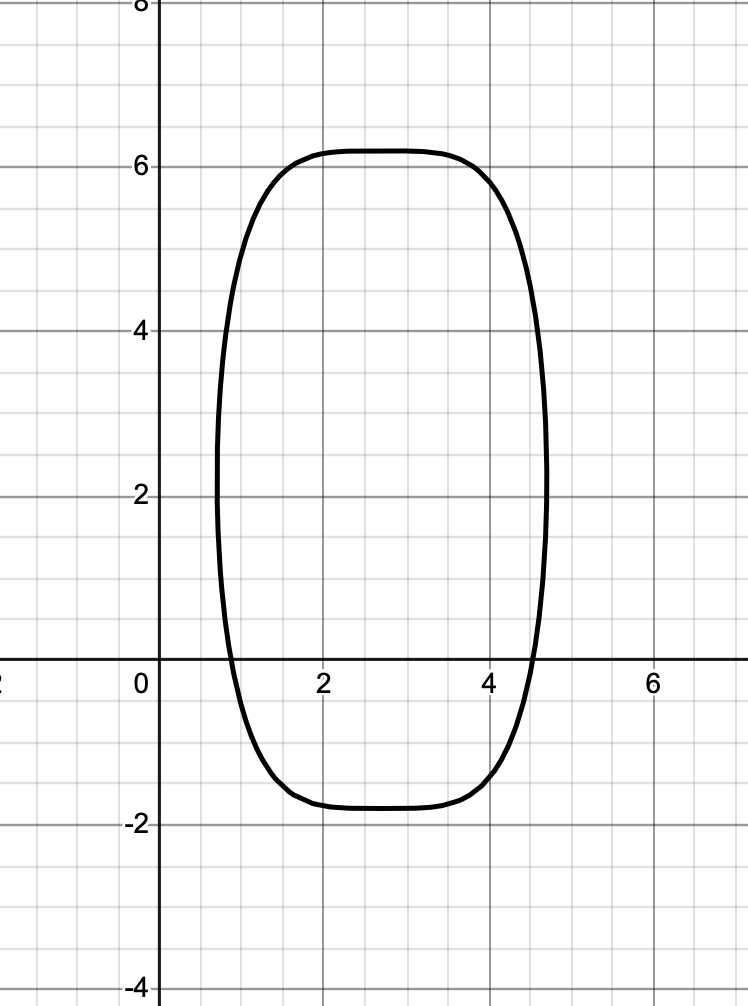}
  \caption{$x=(3,0,3)$, $y_2=x_2$ }
\end{figure}
\begin{figure}[H]
  \centering
\includegraphics[scale=.25]{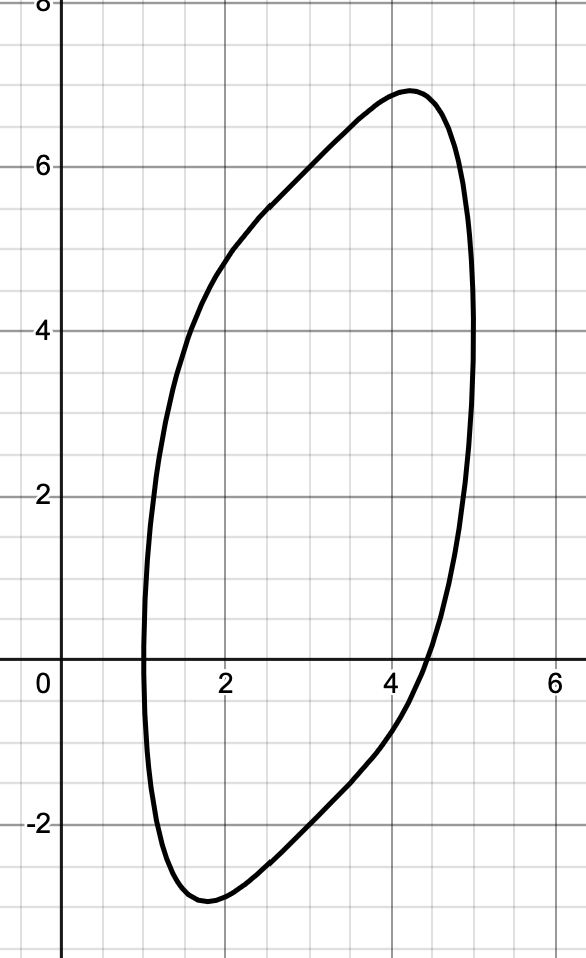}
  \caption{$x=(3,2,2)$, $y_2=x_2$}
  \end{figure}
\end{multicols}
\centerline{
  Level sets of 
$\{ \phi_4(x,y) =R\}$ when $D=3$ and $R=2$ 
}
\vs

Let $\alpha\in \mathbb{N}$ with $\alpha\geq 2$. The Kor\'anyi norm is a special case of the family of anisotropic gauge functions on Heisenberg groups given by
\begin{equation}\label{HeisenbergNorm}
\Vert x \Vert_{\alpha} := \big(\abs{\ux}^{\alpha}+C_{\alpha}\abs{\ox}^{\alpha/2}\big)^{1/\alpha},
\end{equation}
where $C_{\alpha}>0$ is a constant chosen so that the function above satisfies the triangle inequality with respect to the Heisenberg group law. The corresponding Heisenberg norm ball, centered at the origin and of radius $R$ is given by
\begin{equation}\label{Hballgen}
B_R^{\alpha}(\vec{0}) := \big\{y=(\uy,\oy) \in \mathbb{R}^{2n} \textsf{ x } \mathbb{R}: |\uy|^\alpha + 16|\oy|^{\alpha/2} \leq R^\alpha\big\}.
\end{equation}
For $D= 2n+1$, define 
$$\phi_\al: {\mathbb{R}}^D \times {\mathbb{R}}^D \to {\mathbb{R}}$$
by 
\begin{equation}\label{phi}
\phi_\al(x,y) = \| x * y^{-1} \|_\alpha 
= \left( \, |\underline{x} - \underline{y}|^\alpha + C_{\alpha}|\ox -\oy + \frac{1}{2}\underline{y}^TJ \underline{x}|^{\alpha/2}  \, \right)^{1/\alpha}.
\end{equation}

%CHECK HOMOGENEITY 
It is simple to check that $\phi_\al$ satisfies the following quasihomogenity condition for $q>0$:
$$\phi_\al(
q x_1,\cdots, q x_{2n}, q^{2}x_D, 
q y_1,\cdots, q y_{2n}, q^{2}y_D )=q\phi_\al(x,y).$$ 
\vs

The Heisenberg ball of degree $\alpha$ \textbf{centered at $x\in \mathsf{H}_n$} and of radius $R$ is then given by
\begin{equation}
    \label{def hberg ball}
    B^{\alpha}_R(x):=\{y\in \mathsf{H}_n: \phi_{\alpha}(x, y)\leq R\}.
\end{equation}
Note that when $x=\vec{0}$, we recover the Heisenberg ball of degree $\alpha$ centered at the origin, as given by \eqref{Hballzero}.

We shall consider not just a single Heisenberg ball, but a family of balls centered around a truncated, non-isotropic lattice in the Heisenberg group. 
To this effect, for $Q\in \N$ large, $D\in \N$, and $c>0$ a fixed constant, 
consider the truncated lattice consisting of approximately $Q^{D+1}$ points: 
\begin{equation}\label{trun lat}
 L_{D,Q}:= \big\{(b_1, \dots, b_D) \! : \! |b_i| \in \{0,1, \dots, c\,Q  \}, i \neq D, \!\text{ and } |b_D| \in \{0,1, \dots, c\,Q^2 \} \!\big\} 
\end{equation}
$$\subset Z^D \times Z^D. $$

\subsection{Main Results}\label{main section}
Our main result is an upper bound on the average number of lattice points in a $\delta$-neighborhood of Kor\'anyi spheres with centers lying in the truncated lattice $L_{D, Q}$.

\begin{theorem}[Main result for the Kor\'anyi sphere $\alpha=4$] \label{maingeneral} 
Let $n\geq 1$ and set $D=2n+1$.
%and $\al\geq 2 $ be integers, with $\alpha$ being even. 
For $\de\in (0,1)$, $Q\in \N$, and $\phi_4$ as in \eqref{phi} (with $\alpha=4)$, we have
%$\delta \le q^{\beta-\frac{n}{n-r(\al)}}$,
\begin{equation} \label{maineqgeneral}  
Q^{-(2n+2)}\# \{(u,v) \in  L_{D,Q}\times L_{D,Q}:   | \phi_4(u,v)- Q|\leq \delta \}
\end{equation} 
$$\lesssim 
\max \{
Q^{2n},
Q^{2n+1}\delta
\}. $$
Here the implicit constants depend only on $\phi_\alpha$ and $D$. 
\end{theorem}
In \S\ref{ss compare}, we provide a detailed comparison of the above bound with the previously known estimates. For now, we mention that for $n=1$ and $2$, the above result saves a logarithm over the bound 
$$Q^{-(2n+2)}\# \{(u,v) \in  L_{D,Q}\times L_{D,Q}:   | \phi_4(u,v)- Q|\leq \delta \}\lesssim 
\max \{
Q^{2n}\left(\log Q\right)^{\frac{2}{n+1}},
Q^{2n+1}\delta
\},$$
implied by the main result in \cite{GNT} for lattice points on a fixed Kor\'anyi sphere centered at the origin, and matches their bound in higher dimensions (for $n\geq 3$). Further, it follows by the work of Gath \cite{GathSharp} that for $n=1$, the estimate \eqref{maineqgeneral} is the best one possible, if one approaches the problem by bounding lattice points lying on a fixed sphere and then average over the translations. 
However, Theorem \ref{maingeneral} establishes this estimate directly, using techniques which are considerably simpler than the ones employed in \cite{GNT, Gath2}. 

\vs
When $\alpha=2$, our method again gives an upper bound for the average number of lattice points near Heisenberg norm spheres which matches the one implied by the main result in \cite{GNT}. 
However, for $\alpha\geq 6$, our estimate is weaker than the one obtained from \cite{GNT} given by
$$Q^{-(2n+2)}\# \{(u,v) \in  L_{D,Q}\times L_{D,Q}:   | \phi_{\alpha}(u,v)- Q|\leq \delta \}\lesssim 
\max \{
Q^{2n},
Q^{2n+1}\delta
\},$$
although our bound asymptotically approaches the above as $n\to \infty.$

\begin{theorem}[Main result $\alpha\neq 4$] \label{maingeneral not4} 
Let 
$n\geq 1$ and $\al\geq 2 $ be integers with $\alpha\neq 4$, with $\alpha$ being even. Set $D=2n+1$.
For $\de\in (0,1)$, $Q\in \N$, and $\phi_\alpha$ as in \eqref{phi}, we have
\begin{equation} \label{maineqgeneral2}  
Q^{-(2n+2)}\# \{(u,v) \in  L_{D,Q}\times L_{D,Q}:   | \phi_\al(u,v)- Q|\leq \delta \}
\end{equation} 
$$\lesssim \begin{cases}
  \max \{
Q^{2n},
Q^{2n+1}\delta
\}, &\textrm{for } \alpha=2.\\
 \max \{
Q^{2n +\, \frac{2}{D}},
Q^{2n+1}\delta
\}, &\textrm{for } \alpha\geq 6.\\
\end{cases}
$$
Here the implicit constants depend only on $\phi_\alpha$ and $D$. 
\end{theorem}

\subsection{Background and Previous Work}
In this section, we outline some related background results in the literature.  In \S \ref{ss compare}, we compare Theorems \ref{maingeneral} and \ref{maingeneral not4} to the existing work on lattice point problems in Heisenberg groups.  
\vs

\subsubsection{The Euclidean setting}\label{Euclidean} Counting integer lattice points on and near convex surfaces is a classical problem in number theory and related areas. 
The celebrated Gauss circle problem inquires about the number of lattice point within a planar disc centered at the origin of radius $t$. This number is approximated by the area of the disc, and the challenge is to estimate the error. 
We refer the reader to \cite{Hux96} and the references therein for a detailed introduction, and to \cite[\S8.8]{SandS4} for connections to oscillatory integrals. 

More generally, 
let $\Omega$ denote a convex domain in $\R^D$, and set 
$$N_\Omega(t) = card(t\Omega \cap \mathbb{Z}^D).$$
It is not hard to see that $N_\Omega(t)$ is asymptotic to $t^D vol(\Omega)$ as $t\rightarrow \infty$ and that the error term 
$$E_\Omega(t) = N_\Omega(t)-t^D vol(\Omega)$$ is $O(t^{D-1})$. 
When the boundary of $\Omega$ has suitable curvature properties, then the error term improves.  In particular, if the Fourier transform of the surface measure on $\partial \Omega$ 
exhibits decay of $\widehat{\sigma}(\xi) = 
O(\left| \xi \right|^{-\alpha})$, then 
$E_\Omega(t) = O(t^{D-1-\frac{\alpha}{d-\alpha}})$ (see
\cite[\S1.3]{ISS} and the references therein).

A general result of Lettington from \cite{L10} states that if $B\subset \R^D$, $D\geq 2$, is a symmetric convex body with a strictly
convex boundary, then
\begin{equation} \label{key} \# \{k \in {\mathbb{Z}}^D: Q \leq {||k||}_B \leq Q+\delta \} \lesssim C \max \{ Q^{D-1-\frac{\left(\frac{D-1}{2}\right)}{D-\left(\frac{D-1}{2}\right)}}, Q^{D-1} \delta \} \end{equation}
where 
$${||x||}_B=\inf \{t>0: x \in tB \}.$$  
Lettington further requires $\partial B$ to have a tangent hyper-plane at every point and that any two-dimensional cross section through the normal consists of a plane curve with continuous radius of curvature bounded away from zero and infinity. 
If the boundary of $B$ is smoother, some improvements to Lettington's bound are discussed in \cite{IT}. 

Note that the Heisenberg group falls outside of the scope of Lettington's result because the Heisenberg balls do not satisfy the geometric conditions due to local flatness near the poles. Indeed, the surface of $B_Q^{\alpha} $ has points with curvature vanishing to maximal order, meaning all principle curvatures vanish at that point; more details are found in \cite{CT}.    
\vs

In general, a more complex problem arises when we consider lattice point counting problems for surfaces with points of vanishing curvature. In addition to the Heisenberg norm spheres, this 
 includes error estimates in the case of super spheres \cite{Krtl3} and $\ell^p$-spheres \cite{Randol1}, as well as more general surfaces of rotation \cite{KN1, Nowak}.  
See also \cite{Peter1} where the effect of points of vanishing curvature on the error term or ``the lattice remainder term'' is studied in much greater generality.  

\subsubsection{The variable coefficient, multi-parameter setting} 
\label{sssec IT}
In \cite{IT}, Iosevich and the second author proved a variable-coefficient extension of Lettington's result using the theory of generalized Radon transforms. Further, they also established a multi-parameter version applicable to certain non-isotropic dilation groups. 

More precisely, let
$\phi: {\mathbb{R}}^D \times {\mathbb{R}}^D \to {\mathbb{R}}$ denote a
$C^{\lfloor \frac{D}{2} \rfloor+1}$ function 
satisfying the quasi-homogeneity condition
\begin{equation}
    \label{conditionsIT1}
    \phi(q^{\alpha_1}x_1,\cdots, q^{\alpha_D}x_D, q^{\alpha_1}y_1,..., q^{\alpha_D}y_D)=q^{\beta}\phi(x,y),
\end{equation}
where
\begin{equation}\label{conditionsIT2}
\sum_{j=1}^D\alpha_j=D,\,\, \alpha_j >0,\,\, \beta\geq 1, \text{ and }  
\max_j{\alpha_j} \le \frac{2D}{(D+1)}.
\end{equation}
Suppose that 
\begin{equation} \label{eq grad nzero} \nabla_x \phi(x,y)\neq \vec{0} \text{ and } \nabla_y \phi(x,y)\neq \vec{0} \end{equation} in a neighborhood of the sets 
$$\{x \in B: \phi(x,y)=t\}, \textrm{ and } \{y \in B: \phi(x,y)=t\},$$ where, $B$ denotes the unit ball in $\mathbb{R}^D$ with respect to the Euclidean norm.   
Further, suppose that the Monge-Ampere determinant
of $\phi$ (introduced by Phong and Stein in \cite{PhongStein}), given by
\begin{equation}
    \label{eq MongeAmp}
    \textrm{det}
    \begin{pmatrix}
        0& \left(\nabla_y\phi\right)^{\intercal}\\
        \nabla_x\phi & \frac{\partial^2\phi}{\partial_x\partial_y} 
    \end{pmatrix}
\end{equation}
does not vanish on the set $\{(x,y) \in B \times B: \phi(x,y)=t \}$ for any $t>0$.  
The main result of \cite{IT} (Theorem 1.3) then proves that
$$
q^{-D}\# \{(n,m) \in {\mathbb{Z}}^D \times {\Bbb Z}^D: \forall j, |n_j|, |m_j| \leq Cq^{\alpha_j}; | \phi(n,m)- q^{\beta}|\leq \delta \}
$$
\begin{equation}\label{mainIT}
\lesssim \max \{q^{D-2+\frac{2}{D+1}}, q^{D-\beta} \delta \}.
\end{equation}
Note that our definition of the Monge-Ampere determinant is superficially different from that used in \cite{IT} as we have switched the positions of $\nabla_x\phi$ and $\nabla_y \phi$. However, this does not change the value of the determinant. Our definition is the same as the one in the standard reference \cite{Stein}[Chapter XI, \S 3.1].
\vs

In Section \ref{sec rank calc} (proof of Proposition \ref{prop rank}), we will show that $\phi_\alpha$ in \eqref{phi} satisfies the Monge-Ampere determinant condition only when $\alpha=4$ or $\alpha=2$.
This property was used in \cite{sri24} to establish sharp $L^p$ improving estimates for the operator associated to averages over the Kor\'anyi sphere (also see \cite{OSch98}). We will also show that when $\alpha\geq 6$, the rank of the matrix in \eqref{eq MongeAmp} drops by one. The maximal rank of the Monge-Ampere matrix has consequences for the $L^2$ -Sobolev mapping property of the associated averaging operators. These properties, combined with the energy estimates established in \cite{IT} and \cite{CT}, lead to the desired upper bounds for our counting functions. 

\subsubsection{Lattice point counting on Heisenberg groups}
The Heisenberg groups are, in a way, the ``most Euclidean" examples of non-abelian, non-compact Lie groups. There is a lot of activity around proving Heisenberg analogues of results from diverse areas of Euclidean harmonic analysis. For example, a number of recent results investigate mapping properties of maximal averaging operators associated with spheres on Heisenberg groups (see \cite{sri24, JSS21, BHRT, GT} and the references therein). 

%GNT 
In \cite{GNT}, Garg, Nevo, and the second listed author initiated the study of a variant of the Gauss circle problem replacing Euclidean balls by the Heisenberg norm balls defined in \eqref{Hballzero}, but fixed to be centered at the origin.
Specifically, the main result in \cite{GNT} provides bounds on the error term, 
\begin{equation}\label{error}
|E_\alpha(Q)| := | |B_Q^\alpha(\vec{0}) | - N_{\alpha}(Q) |,
\end{equation}
where $N_{\alpha}(Q)$ denotes the number of lattice points in and on the ball
$$
N_{\alpha}(Q) = \left\{   (z,t) \in \Z^{2n}\times \Z:  \Vert (z,t) \Vert_{\alpha} = Q \right\},
$$ 
and
$|B_Q^\alpha(\vec{0}) |= Q^{2n+2}|B_1^{\alpha}(\vec{0})|$ denotes the volume with respect to the Haar measure on $\mathsf{H}^n$.

The best known upper bounds on $|E_{4}(Q)|$ are as follows
 \begin{align}\label{GNT main 4}
 |E_4(Q)|& \lesssim 
     \begin{cases}
 Q^{2}\log{Q} &\text{ when } n=1
 \\
 Q^{4}\log{Q}^{2/3}  &\text{ when } n= 2
 \\
 Q^{2n-\frac{2}{3}} &\text{ when }  n\geq 3,
 \\
     \end{cases} \nonumber\\
 \end{align}
where the implicit constants depends on $n$ but are independent of $Q$.
The bound on $|E_4(Q)|$ when $n\geq 3$ is due to Gath \cite{Gath2}, who significantly improved the previously known bound of $Q^{2n}\log{Q}$ from \cite[Theorem 1.1]{GNT}).
Further, these bounds are sharp when $n=1$ (see \cite{GathSharp}).

For more general $\alpha$, the best known upper bounds on $|E_{\alpha}(Q)|$ are

 \begin{align}\label{GNT main}
 |E_\alpha(Q)|& \lesssim 
     \begin{cases}
 Q^{2n} &\text{ when } \alpha=2
 \\
 Q^{2}\log{Q} &\text{ when } 4 \geq \alpha \geq 3\text{ and } n=1
 \\
 Q^{4}\log{Q}^{2/3}  &\text{ when } \alpha \geq 3 \text{ and } n= 2
 \\
 Q^{2n}  &\text{ when } \alpha \geq 3, \alpha \neq 4, \text{ and } n\geq 3
 \\
 Q^{2n-\frac{2}{3}} &\text{ when } \alpha = 4 \text{ and }   n\geq 3
 \\
  Q^{2 +\delta(\alpha)}\log{Q} &\text{ when }  \alpha >4 \text{ and } n=1
 \\
 Q^{2n} &\text{ when }
 \alpha>4 \text{ and } n\geq 3,
     \end{cases} \nonumber\\
 \end{align}
 where
$\delta(\alpha) = \frac{ 2(\alpha-4)}{3\alpha-4}$. Here, the implicit constants depends on $\alpha$ and $n$ but are independent of $Q$.
These bounds are sharp when $\alpha=2$ in all dimensions (see \cite{GNT}) and when $\alpha = 4$ and $d=1$ (see \cite{GathSharp}).

The method of bounding the error term in the lattice point counting problem in \cite{GNT} uses a hybrid of Euclidean and Heisenberg group techniques. 
The main tools are van-der-Corput lemma (from harmonic analysis), Poisson summation formula, and known asymptotic expressions for the Bessel functions. 
In \cite{GNT}, the authors dominate the lattice point count in $B_Q^{\alpha}$ from above and below by the Euclidean convolution $\chi_{B_Q^{\alpha}} \ast \rho_\epsilon$, where $\rho_\epsilon$ is a bump function adapted to Heisenberg dilations, rather than Euclidean ones.  Next, the Euclidean Poisson summation formula is applied to $\chi_{B_Q^{\alpha}} \ast \rho_\epsilon$, and the resulting product is estimated using spectral decay estimates. Estimates on the the Euclidean Fourier transform of the characteristic function of $B_Q^{\alpha}$ play a crucial role.

In that same work, the authors also compare the upper bound obtained on the error term with the lower bound that is obtained from slicing - namely by viewing their lattice point problem in each hyperplane $(\uy, \oy)$  with $\uy$ fixed separately. In the hyperplane, they apply the known results regarding the classical Euclidean sum-of-squares problem in $\Q^{2n}$.  For the first Heisenberg group (when $n=1$), slicing is inferior to the method described above (with the exception of sufficiently large $\alpha$) \cite[$\S$ 5.2.2]{GNT}. On the other hand, for the higher-dimensional Heisenberg groups ($n\geq 2$), the method of slicing is surprisingly effective and yields the same results as the technique above when $2<\alpha\le 4,$ and slicing yields superior results for large $\alpha$.  The authors point out that the slicing argument utilizes highly non-elementary results for the classical lattice point counting problem in Euclidean balls, and it is interesting to note that just using van-der-Corput lemma and Poisson summation produces the same bound when $2<\alpha\le 4$ \cite[$\S$ 5.3]{GNT}.

Note, the bounds in \eqref{GNT main} can be used to obtain
upper bounds on the number of lattice points 
near the surfaces of the Heisenberg norm balls centered at the origin, as defined in \eqref{Hballzero}.
%Let $n\geq 1$ and $\al\geq 2 $ be integers, and set $D=2n+1$. 
By \cite[Proposition 5.2]{CT}, for instance, we 
have
\begin{equation}\label{Prop5.2}
 %\mathcal{E}_{\alpha}^{0}(Q):=
 \#\big( \{ m \in \mathbb{Z}^{D} :  
Q-\delta \leq \Vert m \Vert_\alpha \leq Q+ \delta \}\big) 
 \lesssim 
 \max\{|E_\alpha(Q)|, Q^D\delta \}.
 \end{equation}
Plugging the estimates from \eqref{GNT main} into \eqref{Prop5.2} yields the best-known estimates on the number of lattice points in a $\delta$-shell about $B_Q^\alpha(\vec{0})$.

Before ending this section, we note that in \cite{CT}, Campolongo and the second listed author provide upper bounds on the number of lattice points near the surfaces of the Heisenberg norm balls directly using an alternative proof technique. 
 The presence of points of vanishing curvature on the spheres is tackled there by obtaining asymptotic bounds on the Fourier transform of supporting surface measures.

\begin{table}[h!]
    \centering
    \caption{A Comparison for Bounds on ${E}_4(Q)$}
    \label{tab:my_label}
\begin{tabular}{|l|c|c|c|}
\hline
\textbf{$n$} & $\mathbf{1}$ & $\mathbf{2}$ & $\geq \mathbf{3}$ \\ \hline
\textbf{Garg-Nevo-Taylor \cite{GNT}}   & {$Q^2\,(\log Q)$}          & {$Q^4\,(\log Q)^{\frac{2}{3}}$}            &      $Q^{2n}$      \\ \hline
\textbf{Gath \cite{Gath2}}   & -            & -            &  $Q^{2n-\frac{2}{3}}$            \\ \hline
% \textbf{Camplongo-Taylor \cite{CT}}   & $Q^{2+\frac{2}{3}}$           & $Q^{4+\frac{6}{7}}$            & $Q^{2n+\frac{4n-2}{4n+1}}$            \\ \hline
% \textbf{Theorem \ref{maingeneral}}   & $Q^2$           & $Q^4$          & $Q^{2n}$           \\ \hline
\end{tabular}    
\end{table}

\subsection{Comparison with Previous Results}\label{ss compare}
 To put our main results into context, we give some commentary and comparisons to the known literature outlined in the previous section.  We begin by observing a simple trivial bound.
\vs
\subsubsection{The trivial bound}
For large values of $Q$, 
the left-hand-sides of \eqref{maineqgeneral} and \ref{maineqgeneral2} are bounded trivially by 
\begin{equation}\label{trivial}
\left|B_{(Q+2)}^\alpha\right| - \left|B_{(Q-2)}^\alpha \right|,\end{equation}
where 
$|\cdot |$ is used to denote the Euclidean volume, and $B_Q^\alpha$ is defined in \eqref{def hberg ball}. 
Recalling that the homogeneous dimension of $\mathsf{H}_n$ is $2n+2$, it follows that $|B_Q^\alpha|=Q^{2n+2}|B_1^\alpha|$, and so
the expression in \eqref{trivial} is bounded trivially by a constant multiple of $Q^{2n+1}$, and our bound is an improvement for all $\al$.  

\subsubsection{Results on Heisenberg Lattice Point Counting Problems}
Note that bounds on the number of lattice points near/on dilates of a fixed sphere centered at the origin can be leveraged to obtain bounds on the average number of lattice points over translates of the sphere. To see this, for $Q>0$, define $\mathcal{E}_{\alpha}^{\textrm{avg}}(Q)$ to be the smallest number such that the following inequality holds true
\begin{equation}
    \label{eq err avg}
    Q^{-(2n+2)}\# \{(u,v) \in  L_{D,Q}\times L_{D,Q}:   | \phi_\al(u,v)- Q|\leq \delta \}\lesssim \max\{\mathcal{E}_{\alpha}^{\textrm{avg}}(Q), Q^{D}\delta\}.
\end{equation}
We first apply triangle inequality and then use translation invariance of the lattice point counting problem to estimate
\begin{align*}
& Q^{-(2n+2)}\# \{(u,v) \in  L_{D,Q}\times L_{D,Q}:   | \phi_\al(u,v)- Q|\leq \delta \}\\  &\leq   Q^{-(2n+2)} \sum_{u\in L_{D,Q}} \# \{v\in  L_{D,Q}:   | \phi_\al(u,v)- Q|\leq \delta \}\\
&= Q^{-(2n+2)} \sum_{u\in L_{D,Q}} \# \{v\in  L_{D,Q}:   | \phi_\al(u*v^{-1},\vec{0})- Q|\leq \delta \}.
\end{align*}
Recalling that $\phi_\al(u*v^{-1},\vec{0})=\|u*v^{-1}\|_{\alpha}$, the expression above is equal to
\begin{align*}
& Q^{-(2n+2)} \sum_{u\in L_{D,Q}} \# \{v\in  L_{D,Q}: Q-\delta\leq \|u*v^{-1}\|_{\alpha} \leq Q+\delta \}\\
&\lesssim Q^{-(2n+2)}Q^{(2n+2)} \max_{u\in L_{D,Q}}\# \{v\in  L_{D,Q}: Q-\delta\leq \|u*v^{-1}\|_{\alpha} \leq Q+\delta \}.
%&\lesssim \max\{\mathcal{E}_{\alpha}^{\textrm{avg}}(Q), Q^{D}\delta\},
\end{align*}

We can now apply \eqref{Prop5.2} to bound the last term above and conclude that
\begin{multline*}
 \label{eq err prev}
    Q^{-(2n+2)}\# \{(u,v) \in  L_{D,Q}\times L_{D,Q}:   | \phi_\al(u,v)- Q|\leq \delta \}\lesssim 
    Q^{D}\delta+\min \{|E_{\alpha}(Q)|, \mathcal{E}^0_{\alpha}(Q)\}.  
\end{multline*}
In other words, 
\begin{equation}
    \label{eq dom avg fixed}
    \mathcal{E}_{\alpha}^{\textrm{avg}}(Q)\leq \min \{|E_{\alpha}(Q)|, \mathcal{E}^0_{\alpha}(Q)\},
\end{equation}
where $E_\alpha(Q)$ and $\mathcal{E}_{\alpha}^{0}(Q)$ is defined in \eqref{error}. 
% and \eqref{eq def E CT} respectively. 
\smallskip

Now, if we use Theorem \ref{maingeneral} and Theorem \ref{maingeneral not4} to bound the term on the left hand side in \eqref{eq err avg}, we obtain the upper bound
\begin{equation}
    \label{eq err mainthm}
    \mathcal{E}_{\alpha}^{\textrm{avg}}(Q)\leq 
    \begin{cases}
        Q^{2n}, & \textrm{ for } \alpha=2 \textrm{ and } 4.\\
        Q^{2n+\frac{2}{D}}, & \textrm{ for } \alpha\geq 6.
    \end{cases}
\end{equation}

\begin{table}[h!]
    \centering
    \caption{A Comparison for Bounds on $\mathcal{E}^{\textrm{avg}}_4(Q)$}
    \label{tbl eavg comp}
\begin{tabular}{|l|c|c|c|}
\hline
\textbf{$n$} & $\mathbf{1}$ & $\mathbf{2}$ & $\geq \mathbf{3}$ \\ \hline
\textbf{Garg-Nevo-Taylor \cite{GNT}}   & {$Q^2\,(\log Q)$}          & {$Q^4\,(\log Q)^{\frac{2}{3}}$}            &      $Q^{2n}$      \\ \hline
\textbf{Gath \cite{Gath2}}   & -            & -            &  $Q^{2n-\frac{2}{3}}$            \\ \hline
\textbf{Campolongo-Taylor \cite{CT}}   & $Q^{2+\frac{2}{3}}$           & $Q^{4+\frac{6}{7}}$            & $Q^{2n+\frac{4n-2}{4n+1}}$            \\ \hline
\textbf{Theorem \ref{maingeneral}}   & $Q^2$           & $Q^4$          & $Q^{2n}$           \\ \hline
\end{tabular}    
\end{table}

A few remarks are in order to compare our results, Theorems \ref{maingeneral} and \ref{maingeneral not4}, with the ones described above.

\begin{remark}
    Theorem \ref{maingeneral}, for $\alpha=4$, shall be established following the same philosophy as the variable coefficient lattice point counting estimate in \cite{IT} (see \S \ref{sssec IT}). Indeed, the verification of the homogeneity conditions \eqref{conditionsIT1}- \eqref{conditionsIT2} is straightforward. The main work lies in establishing that the Monge-Ampere determinant of $\phi_{4}$, as defined in \eqref{eq MongeAmp}, does not vanish away from the origin. The fact that the definition of $\phi_{4}$ incorporates the Heisenberg group law (manifested in the presence of the bilinear form $J$ in \eqref{phi}) into account, is crucial to this proof. This presence of ``rotational curvature" has been established before in the first author's work \cite{sri24}, and in the unpublished thesis \cite{OSch98} of O. Schmidt. In the former, it was used to establish optimal $L^p$ improving estimates for averaging operators associated to the Kor\'anyi spheres. However, to the best of our knowledge, the current work is the first to leverage this observation to derive estimates on lattice points near translates of the Kor\'anyi sphere. The work \cite{IT} did not consider application to lattice point counting problems on Heisenberg groups. On the other hand, works like \cite{GNT, Gath2, CT} which deal with these problems, have not investigated the effect of group translations on the estimates for the counting function. In this way, Theorem \ref{maingeneral} is a true analogue of the Euclidean lattice counting problems in the Heisenberg setting.
\end{remark}

\begin{remark}
    In contrast, if one considers lattice points near Heisenberg spheres translated with respect to the Euclidean group law, the Monge-Ampere determinant of the corresponding describing function does vanishes at points away from the origin. Indeed, in this case, the matrix in \eqref{eq MongeAmp} can have rank zero, which is a consequence of the fact that the Kor\'anyi sphere centered at the origin has zero Gaussian curvature at the poles. This is reason why Theorem \ref{maingeneral} yields better estimates for $\mathcal{E}_{\alpha}^{\textrm{avg}}(Q)$ than the main result of Theorem \cite{CT}, where the authors spend considerable effort to overcome this ``lack of curvature" (see Table \ref{tbl eavg comp}). 
\end{remark}

\begin{remark}
    Our method also allows us to save a logarithm over the bounds on $\mathcal{E}_{4}^{\textrm{avg}}(Q)$ implied by the main result in \cite{GNT} for $n=1$ and $2$, and matches their bound in higher dimensions (for $n\geq 3$). Note that by the work of Gath \cite{GathSharp}, we know that in the case when $n=1$, we have $|{E}_{4}(Q)|\gtrsim Q^{2}$. 
    Thus a sharp estimate for $|{E}_{4}(Q)|$ would lead to the bound
    $$\mathcal{E}_4^{\textrm{avg}}(Q)\lesssim Q^2,$$
    which is what Theorem \ref{maingeneral} establishes directly. However, our method is considerably simpler than the techniques employed in \cite{GNT, Gath2}.
\end{remark}

\begin{remark}
When $\alpha\neq 4$, our Theorem \ref{maingeneral not4} can be thought of as an extension 
of the main result in \cite{IT}, for functions $\phi_{\alpha}$ with vanishing Monge-Ampere determinant, but for which the matrix in \eqref{eq MongeAmp} still has non-zero rank. Indeed, we show that in these cases, the said matrix loses rank by one. Correspondingly, the underlying generalized Radon transform has a slightly worse Sobolev mapping property, which is reflected in the sub-optimal bound on the error term $E_{\alpha}^{\textrm{avg}}(Q)$, as compared to the estimates in \cite{GNT} for $\alpha\neq 4$, which uses a deeper analysis involving Bessel functions. However, it might be possible to improve Theorem \ref{maingeneral not4} in certain regimes by exploring the finer mapping properties of the canonical relation associated to the underlying oscillatory integral operator. 
\end{remark}

%%%%%%%%%

\subsection{Structure of the Paper}
In Section \ref{sec outline pf}, the proof of the main theorems is reduced to bounding the $L^2$ norms of two terms, the first resembling an energy integral and the second an averaging operator over Heisenberg spheres, acting on a measure supported on a scaled, truncated lattice. The first term is handled in Section \ref{boundI} using pre-established energy estimates from \cite{IT, CT, Campo}. The second term is estimated in Section \ref{sec boundII}, by establishing the precise $L^2$ Sobolev mapping property of the underlying averaging operator. This depends on the rank of the associated Monge-Ampere matrix, which is calculated in Section \ref{sec rank calc} following the arguments in \cite{sri24}.

\subsection{Notation}
By $A\lesssim B$ we shall mean that $A\le C\cdot B,$ where $C$ is a positive constant and $A\sim B$ shall signify that $A\lesssim B$ and $B\lesssim A$. $A\lesssim_D B$ shall mean that $A\leq C\cdot B$ with the positive constant $C$ depending on the parameter $D$.

\subsection{Acknowledgement}
R.S. is
supported by the Deutsche Forschungsgemeinschaft (DFG, German Research Foundation) under Germany’s Excellence Strategy 
- EXC-2047/1 - 390685813, by SFB 1060 and by an Argelander grant of the University of Bonn. K.T. is supported in part by the Simons Foundation Grant GR137264. Both of them thank Professor Allan Greenleaf for helpful discussions.

\section{Proof of Theorems \ref{maingeneral} and \ref{maingeneral not4}}
\label{sec outline pf}
In this section, we use a scaling argument to transform our lattice point counting problem into a measure-theoretic estimate.  Using the Fourier inversion formula, we re-write and bound the arising quantity by two terms, $I$ and $II$.  The first term resembles an energy integral, a well-known notion in the field of geometric measure theory, and is handled using pre-established bounds from \cite{IT} (when $\alpha=4$) and from \cite{CT} (when $\alpha\neq 4$). To bound the second term, we need to establish the Sobolev mapping property $L^2\to L^2_s$ of an underlying averaging operator for a precise value of $s$. This depends on the rank of its Monge-Ampere matrix, which is calculated in Section \ref{sec rank calc} following the arguments in \cite{sri24}[Section 6].

\subsection{Main Counting Lemma}
We use the notation of Section \ref{main section}.  
In particular, recall 
$n\geq 1$ and $\al\geq 2 $ are integers, with $\alpha$ being even, $D=2n+1$, 
$\phi_\alpha$ is the function in \eqref{phi}, and $L_{D,Q}$ is the truncated lattice defined in \eqref{trun lat}. To treat both Theorems \ref{maingeneral} and \ref{maingeneral not4} simultaneously, we introduce the parameter
\begin{equation}
    \label{eq gam def}
    r=r(\alpha):=\begin{cases}
        \frac{D-1}{2} & \text{ for } \alpha=2 \textrm{ and } 4.\\
        \frac{D-2}{2} & \text{ for } \alpha\geq 6.
    \end{cases}
\end{equation}

We rely on the following construction, which is similar to that in \cite{CT} and \cite{IT}.  
Recall $D=2n+1$ and consider the truncated lattice:
\begin{equation}\label{lattice}
 L_{D,q^a}:= \big\{(b_1, \dots, b_D) \! : \! b_i \in \{0,1, \dots, \lceil q^a \rceil \}, i \neq D, \!\text{ and } b_D \in \{0,1, \dots, \lceil q^{2a} \rceil \} \!\big\},
\end{equation}
where $a$ is chosen so that $L_{D,q^a}$ consists of $q^D$ lattice points; namely, 
$a(D-1) + 2a = D$ or 
$$a := \frac{D}{D+1}.$$

 With this set up in place, we prove the following result from which the main theorems follow. 
\vs
\begin{lemma}[Main Counting Lemma]\label{main lemma}
Let $q\geq1$ and $a= \frac{D}{D+1}$. 
%$\delta \le q^{\beta-\frac{n}{n-r(\al)}}$,
For $r(\al)$ be as in \eqref{eq gam def}, set
\begin{equation}
    \label{eq s def}
    s:=D-r(\al)=\begin{cases}
        \frac{D+1}{2} & \text{ for } \alpha=2 \textrm{ and } 4.\\
        \frac{D+2}{2} & \text{ for } \alpha\geq 6.
    \end{cases}.
\end{equation}
Let $\tau$ be a tunable parameter satisfying 
$\tau \in (a, \frac{(D-1)a}{s-1}]$. 
Then, 
\begin{equation} \label{maineq}  
q^{-D}\# \{(n,m) \in L_{D,q^a}\times L_{D,q^a}:  | \phi_\al(n,m)- q^{a}|\leq q^{a-\tau} \}
\end{equation} 
$$\lesssim 
q^{D-\tau}, $$
%$s=D-r(\al)$, 
%$r(\al)$ is as in Theorem \ref{maingeneral}, 
with the implicit constant depend only on $D$ and the function 
$\phi_\alpha$.
\end{lemma}

As the parameter $\tau$ increases, the neighborhood about the surface decreases, and we obtain a more precise estimate for the number of lattice points on the surface.  The largest choice of $\tau$ permitted by Lemma \ref{main lemma} arises in the cases $\alpha=4$ and $\alpha=2$. Observe that counting lattice points in the thickened surface, as on the left-hand-side of \eqref{maineq}, is equivalent to further ``thickening'' the lattice. So, we
consider the $q^{a-\tau}$-neighborhood of $L_{D,q^a}$ for $\tau>a$.
Denote this thickened lattice set by 
$$\left(L_{D,q^a}\right)_{q^{a-\tau}} = \{x\in \R^D: |x-u| < q^{a-\tau} \text{ for some } u \in L_{D,q^a} \}.$$
For visualization, we replace each $(q^{a-\tau})$-\textit{ball} by a $(q^{a-\tau})$-\textit{box} with the same center (see Figure \ref{latticeScale}).

\subsubsection{Proof of Theorems \ref{maingeneral} and \ref{maingeneral not4} assuming Lemma \ref{main lemma}}
To see that the main results, Theorems \ref{maingeneral} and \ref{maingeneral not4}, follow from Lemma \ref{main lemma}, let $\delta\in (0,1)$, set $\tau_0 = \frac{(D-1)a}{s-1}$, 
and consider two cases: either $\delta<q^{a-\tau_0}$ or  $\delta\geq q^{a-\tau_0}$. 
 Since $\tau \in (a, \tau_0]$, the smallest that $q^{a-\tau}$ can be is $q^{a- \tau_0}$. 
\vs 

%ESTIMATE 1
If $\delta < q^{a- \tau_0}$, then Lemma \ref{main lemma} implies that 
\begin{equation}\label{est1}
q^{-D}\# \{(n,m) \in L_{D,q^a}\times L_{D,q^a}:  | \phi_\al(n,m)- q^{a}|\leq \delta \}
\lesssim q^{D-\tau_0}.
\end{equation}
\vskip.125in

%ESTIMATE 2
If $ \delta \geq  q^{a- \tau_0}$, then we can choose $\tau' \in (a, \tau_0]$ so that $\delta =  q^{a-\tau'}$, and Lemma \ref{main lemma} implies that 
\begin{equation}\label{est2}
q^{-D}\# \{(n,m) \in L_{D,q^a}\times L_{D,q^a}:  | \phi_\al(n,m)- q^{a}|\leq q^{a-\tau'} \}
\lesssim q^{D-\tau'} \leq q^{D-a}\delta.\end{equation}
\vskip.125in

%COMBINE ESTIMATES 
Putting these estimates \eqref{est1} and \eqref{est2} together, we conclude that, for $\de\in (0,1)$,
\begin{equation}  
q^{-D}\# \{(n,m) \in L_{D,q^a}\times L_{D,q^a}:  | \phi_\al(n,m)- q^{a}|\leq \delta \}
 \lesssim \max  \{
 %q^{ D - \frac{(D-1)a}{ (D-r(\al) -1 ) } }, 
  q^{ D - \tau_0 }, 
 q^{D-a} \delta 
 \}. 
 \end{equation}
 \vskip.125in

Finally, setting $Q=q^a$, and recalling that $a=\frac{D}{D+1}$ and $\tau_0 = \frac{(D-1)a}{s-1}$, Theorem \ref{maingeneral} follows from a change of variables. 
It remains to establish Lemma \ref{main lemma}. 

\subsection{Proof of Main Counting Lemma \ref{main lemma}} 
\label{reductSec}
The lemma is proved in several steps.  First, matters are reduced to proving a measuring lemma, mainly Lemma \ref{measuring lemma}. Next, Lemma \ref{measuring lemma} is bounded using elementary properties of the Fourier transform and the Cauchy-Schwartz inequality by a product of two integrals.  These integrals, which we call $I$ and $II$, are handled in Sections \ref{boundI} and \ref{sec boundII} respectively. 
\smallskip

We can re-write the left-hand side of \eqref{maineq} as follows.
Let $N_C(S, \rho)$ be the number of cubes of side-length $\rho>0$ required to cover a set $S$. 
Then

\begin{align}
\label{MeasFSsim0}
& \# \left(\big\{(n,m) \in \mathbb{Z}^D \mathsf{x} \ \mathbb{Z}^D : \Vert n \Vert_{\alpha}, \Vert m \Vert_{\alpha} \!\leq Cq^a, \abs{\phi_\al(n, m) - q^a} \leq q^{a-\tau} \big\}\right)
\\
&\sim N_C \left(\big\{(u, v) \in
\left(L_{D,q^a}\right)_{q^{a-\tau}}  \ \mathsf{x} \ \left(L_{D,q^a}\right)_{q^{a-\tau}} 
 :  \abs{ \phi_\al(u,v)  - q^a} \leq q^{a-\tau} \}, q^{a-\tau}\right).
\label{MeasFSsim1}
\end{align}
\vs

Next, scaling $L_{D,q^a}$ down into the unit box, we set
\begin{equation}\label{EqDef}
E_q:= \bigcup_{(b_1, \dots, b_D) \in L_{D,q^a}} \left\{ R_\tau + \left(\frac{b_1}{q^a}, \dots, \frac{b_{2n}}{q^a}, \frac{b_D}{q^{2a}}\right)\right\},
\end{equation}
where $R_\tau$ denotes the $q^{-\tau} \ \mathsf{x} \cdots \mathsf{x} \ q^{-\tau} \ \mathsf{x} \ q^{-a-\tau}$-rectangular box centered at the origin in $\mathbb{R}^{D}$. 
Note $E_q$ is composed of approximately $q^D$ such rectangles (see Figure \ref{latticeScale}).  
\begin{figure}[ht]
  \centering
  \includegraphics[width=\linewidth]{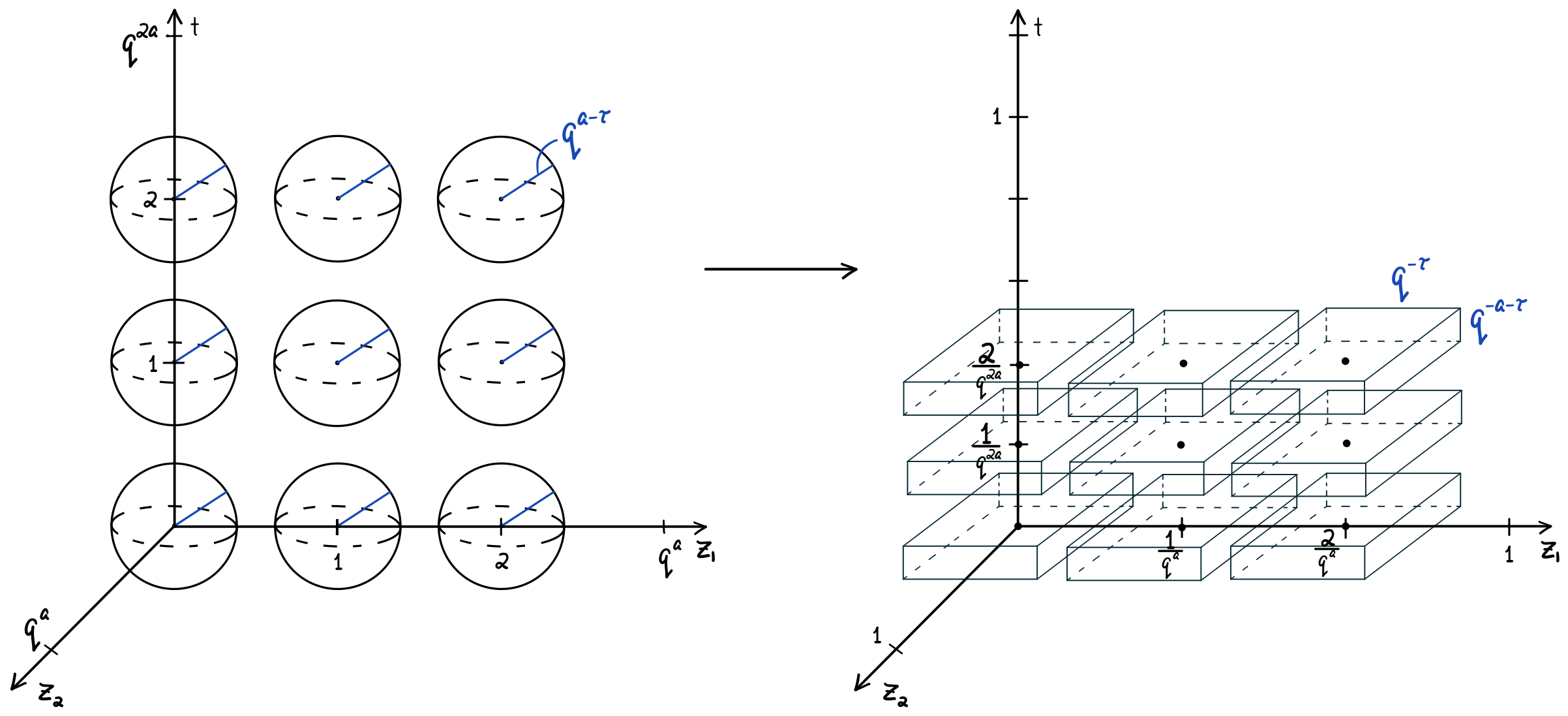}
\caption{The sets $\tau_q^a(E_q)$ (left) and the scaling to $E_q$ (right), where $\tau_q^a$ is the scaling operation in \eqref{scalingop}.}\label{latticeScale}
\end{figure}

Denote the volume of $R_\tau$ by $V_R$ and observe that
\begin{equation} \label{volRect}
V_R := vol(R_\tau) = (q^{-\tau})^{D-1}q^{a-\tau}=q^{-D\tau-a}.
\end{equation}
\vs

We now define a probability measure, $\mu_{q,\tau} = \mu_{q,\tau}$, on $E_q$.  While the definition of $\mu_{q,\tau}$ appears quite complicated, it is simply a normalized and smoothed version of the indicator function of the set $E_q$.
\begin{definition}[Probability measure on $E_q$]\label{defMuQalpha}
For $x \in \mathbb{Z}^{2n+1}=\mathbb{Z}^{D}$ and $\tau>a$, we define $\mu_{q,\tau}$ 
by:
\begin{equation}
\label{def measure}
\mu_{q,\tau}(x) \! := 
\frac{1}{\abs{E_q}} \sum_{b \in \mathbb{Z}^D} \left[\prod_{i=1}^{2n} \psi_0\!\left(\frac{b_i}{q^a}\right)  
\psi_0\!\left(\!q^{\tau}\! \left(x_i-\frac{b_i}{q^a}\right)\!\right)\right] \! \psi_0\!\left(\frac{b_D}{q^{2a}}\right) \psi_0\!\left(\!q^{a+\tau}\! \left(x_D-\frac{b_D}{q^{2a}}\right)\!\right) \!,
\end{equation}
where $\psi_0$ is a bump function supported on the unit ball. 
For $A\subset E_q$, set 
$\int_A d\mu_{q,\tau}(x) = \int_A \mu_{q,\tau}(x) dx.$ 
\end{definition}

The following lemma transforms the
lattice point counting problem to one of bounding the $
\left(\mu_{q,\tau}\times \mu_{q,\tau}\right)$-measure of a scaled-down version of the set.

%%%%%%%%%%%%%%%DECAY TO MEASURE LEMMA
\begin{lemma}[transforming the lattice point counting problem to a measure-theoretic bound]\label{DecayEquivHNB}
\begin{align}
q^{-D}\#\big(\big\{ (n,&m) \in \mathbb{Z}^{D} \ \mathsf{x} \ \mathbb{Z}^{D} : \Vert n \Vert_\alpha, \Vert m \Vert_\alpha \leq Cq^a, \abs{\phi_\al(n, m) - q^a} \leq q^{a-\tau} \big\}\big) \nonumber\\
&\sim q^D \mu_{q,\tau}  \ \mathsf{x} \ \mu_{q,\tau} \left( \big\{ (x,y) \in E_q \ \mathsf{x} \ E_q : \abs{\phi_\al(x,y) - 1} \leq q^{-\tau} \big\} \right) \label{measFS}
\end{align}
\end{lemma}

\begin{proof}
Let $N_C(S, \rho)$ be the number of cubes of side-length $\rho$ and $N_R(S,\rho_1, \rho_2)$ be the number of $\rho_1  \ \mathsf{x} \cdots \mathsf{x} \ \rho_1  \ \mathsf{x}  \ \rho_2$-rectangular boxes required to cover a bounded set $S\subset \R^D$. 
Defining 
$$\tau_q^{a}: \R^D \rightarrow \R^D$$ by 
\begin{equation}\label{scalingop}
\tau_q^{a}(x) =(q^ax_1, \dots, q^ax_{2n}, q^{2a}x_{D}),
\end{equation}
 we can write
 $$\left(L_{D,q^a}\right)_{q^{a-\tau}}  = \tau_q^{a}(E_q).$$
 
 By the equivalence between \eqref{MeasFSsim0} and \eqref{MeasFSsim1}, we have 
\begin{align*}
(& \# \left(\big\{(n,m) \in \mathbb{Z}^D \mathsf{x} \ \mathbb{Z}^D : \Vert n \Vert_{\alpha}, \Vert m \Vert_{\alpha} \!\leq Cq^a, \abs{\phi_\al(n, m) - q^a} \leq q^{a-\tau} \big\}\right)
\\
&\sim N_C \left(\big\{(u, v) \in \tau_q^\alpha(E_q) \ \mathsf{x} \ \tau_q^\alpha(E_q) :  \abs{ \phi_\al(u,v)  - q^a} \leq q^{a-\tau} \}, q^{a-\tau}\right).\\
\end{align*}

Scaling down, we have
\begin{align*}
& N_C \left(\big\{(u, v) \in \tau_q^\alpha(E_q) \ \mathsf{x} \ \tau_q^\alpha(E_q) :  \abs{ \phi_\al(u,v)  - q^a} \leq q^{a-\tau} \}, q^{a-\tau}\right) 
\\
\sim & N_R \left(\big\{(x,y) \in E_q \ \mathsf{x} \ E_q :  \abs{\phi_\al(x,y) - 1} \leq q^{-\tau} \}, q^{-\tau}, q^{-a-\tau}\right).
\\
\end{align*}

We finally conclude that
\begin{align}
& (q^{-D})^2N_R \left(\big\{(x,y) \in E_q \ \mathsf{x} \ E_q :  \abs{\phi_\al(x,y) - 1} \leq q^{-\tau} \}, q^{-\tau}, q^{-a-\tau}\right)
\label{MeasFSsim2}\\
&\sim \mu_{q,\tau} \ \mathsf{x} \ \mu_{q,\tau} \left(\big\{(x,y) \in E_q \ \mathsf{x} \ E_q : \abs{\phi_\al(x,y) - 1} \leq q^{-\tau} \big\}\right),
\label{MeasFSsim3}
\end{align}
where in the last line we have used that $\mu_{q,\tau} \times \mu_{q,\tau} (R_\tau\times R_\tau) = q^{-2D}$ since
$1=\mu_{q,\tau}\times \mu_{q,\tau} (E_q\times E_q) = \mu_{q,\tau} \times \mu_{q,\tau} (R_\tau\times R_\tau) \cdot q^{2D}$.
\end{proof}

 We have now arrived at the heart of the argument and reduced matters to establishing the following estimate.
\begin{lemma}[Measuring Lemma]\label{measuring lemma}
Let $r(\al)$ be as in \eqref{eq gam def}, and set 
\begin{equation}
    %\label{eq s def}
    s:=D-r(\al)=\begin{cases}
        \frac{D+1}{2} & \text{ for } \alpha=2 \textrm{ and } 4.\\
        \frac{D+2}{2} & \text{ for } \alpha\geq 6.
    \end{cases}.
\end{equation}
Then, for 
$\tau \in (a, \frac{(D-1)a}{s-1}]$,
\begin{equation}\label{keyestimate}
\mu_{q,\tau}  \ \mathsf{x} \ \mu_{q,\tau} \left( \big\{ (x,y) \in E_q \ \mathsf{x} \ E_q : \abs{\phi_\al(x,y) - 1} \leq q^{-\tau} \big\} \right) \lesssim q^{-\tau}.
\end{equation}
\end{lemma}

To show \eqref{keyestimate} holds, we begin by re-writing the left hand side of the inequality as
$$\mu_{q,\tau} \times \mu_{q,\tau} (\{ (x,y): |\phi_\al(x,y)-1| \le q^{-\tau}\})
=\iint_{\{|\phi_\al(x,y)-1| \le q^{- \tau } \}}\psi(x,y) \mu_{q,\tau}(y)\mu_{q,\tau}(x)dydx,$$
where $\psi$ is an appropriate compactly supported smooth bump-function with support in a ball centered at origin. 
\vs 

Define
\begin{equation}\label{operator} T_q f(x)= q^{ \tau }\int_{\{|\phi_\al(x,y)-1| \le q^{- \tau } \}}f(y)\psi(x,y) dy.\end{equation}
It follows that the left-hand-side of $(\ref{keyestimate})$ can be written as $q^{- \tau } \left<T_q\mu_{q,\tau} ,\mu_{q,\tau} \right>$, and it remains to show that, 
%for $\frac{n+1}{2}\le s < n$,
$$ \left<T_q\mu_{q,\tau} ,\mu_{q,\tau} \right>\lesssim 1.$$ 
\vs

Recall $r(\alpha)$ from \eqref{eq gam def}. By  the Cauchy-Schwarz inequality:
\begin{equation} \label{heaven}  
\left<T_q\mu_{q,\tau} ,\mu_{q,\tau} \right> 
\le I \times II,
\end{equation}
where
\begin{equation}
    \label{def term 1}
    I:= \| \widehat{\mu_{q,\tau}}(\cdot) <\cdot>^{-\frac{r(\al) \,}{2}})\,  \|_2,
\end{equation}
and
\begin{equation}
\label{def term 2}
 II:= \| \widehat{T_q\mu_{q,\tau}}(\cdot)\times  <\cdot>^{\frac{r(\al)}{2}}) \|_2  
 \end{equation} 
with $<\cdot> = ( 1 + |\cdot|^2)^{1/2}$. 
\vs

To complete the proof of Lemma \ref{main lemma}, it remains to show that the terms $I$ and $II$ are bounded. 

\section{Bounding Term I}\label{boundI}
Here, we show that 
\begin{equation}\label{Bound1}
\int \big|\widehat{\mu_{q,\tau}}(\xi)\big|^2 
\abs{\xi}^{-r(\alpha)}
%<\xi>^{-r(\alpha)} 
d\xi \lesssim 1,
\end{equation}
for $\tau$ in an appropriate range  to be determined momentarily.   
\vs

Using elementary properties of the Fourier transform, 
\begin{equation}\label{energy_2}
 \int_{|\xi|>1} \big|\widehat{\mu_{q,\tau}}(\xi)\big|^2 \abs{\xi}^{-r(\alpha)} d\xi 
\sim \! \iint \abs{x-y}^{r(\alpha)-D} d\mu_{q,\tau}(x)d\mu_{q,\tau}(y).
\end{equation}
See, for instance, \cite{Mattila15} for a discussion of energy-integrals and the Fourier transform. 

We bound the expression in \eqref{energy_2} using the following proposition, which appeared before as Proposition 2.5 in \cite{CT}.

\begin{proposition}[Energy integral bound]\label{energy}
Let $\mu_{q,\tau}$ as given in Definition~\ref{defMuQalpha} and recall $a= \frac{D}{D+1}$.
For $2\leq t\leq\frac{D+2}{2}$,
$$
\iint \abs{x-y}^{-t }d\mu_{q,\tau}(x)d\mu_{q,\tau}(y) \lesssim 1, 
$$
provided 
$
\tau \in \left(a, \frac{(D-1) a }{ t -1} \right].
$
\vs
\noindent
\end{proposition}
We note that the upper bound on $\tau$  will translate to a lower bound on the range of $\delta$ in the main theorem.
Applying Proposition \ref{energy} with $t=D-r(\al)$, with $r(\al)$ as in\eqref{eq gam def}, will complete the estimate and demonstrates that \eqref{Bound1} holds when $\tau \in \left(a, \frac{(D-1) a }{ t -1} \right]$. 

We refer to \cite[Proposition 2.5]{CT}, or \cite[\S 5.2]{Campo} for the proof of Proposition \ref{energy} in the case when $n=1$, and to \cite[\S 5.3]{Campo} in higher dimensions (with the disclaimer that they use the variables $n$ and $d$, to refer to our $D$ and $n$ respectively). The cited estimates as phrased in \cite{CT, Campo} appear to depend on a parameter $\gamma(\alpha)$, analogous to our parameter $r(\alpha)$, with $t=D-\gamma(\alpha)$. However, a quick glance at the proof shows that $\gamma(\alpha)$ play no role in the argument. Indeed, with $a=\frac{D}{D+1}$, one only needs $t$ and $\tau$ to satisfy the conditions
$$2\leq t\leq 3, \qquad \tau\leq \frac{2a}{t-1}$$
when $n=1$ (see condition (5.6) and the discussion right above in \cite{Campo}) for $n=1$;
and analogously, the conditions
$$1<t\leq t_{\alpha}, \qquad \tau\leq \frac{(D-1)a}{t-1}$$
for $n\geq 2$, where
\begin{equation}
    \label{eq camp tal def}
    t_{\alpha}:=\begin{cases}
        \frac{D+2}{2}, & \alpha=2.\\
        D-\min\left\{\frac{D-1}{\alpha}, \frac{1}{2}\right\}, &\alpha\geq 4,
    \end{cases}
\end{equation}
(see condition (5.22) and the discussion preceding it in \cite{Campo}).
Since $t_{\alpha}\geq \frac{D+2}{2}$ for all values of $\alpha$ under consideration, the energy estimates from \cite{Campo} can be applied to our setting.

These energy estimates in \cite{CT, Campo} are inspired by the proof of a similar lemma in \cite{IT}. However, while in \cite{IT}, it is required that the surfaces under consideration satisfy a non-vanishing curvature condition, the methods in \cite{CT, Campo} allow for vanishing curvatures. 
As compared to the proof in \cite{IT}, the argument in \cite{Campo} is much longer and more complex. 
A crude estimate is utilized in \cite{IT}, in which the rectangles - analogous to those introduced above in \eqref{EqDef} - are replaced by larger super-set cubes.   This significantly simplifies bounding the arising energy integral.
\\

\section{Bounding Term II}
\label{sec boundII}
It remains to show that 
\begin{equation}\label{Tbound}  \| \widehat{T_q\mu_{q,\tau}}(\cdot)\times  <\cdot>^{\frac{r(\al)}{2}}) \|_2 \lesssim 1,\end{equation}
where $T_q$, $\mu_{q,\tau}$ and $r(\alpha)$ are as defined in \eqref{operator}, \eqref{def measure} and \eqref{eq gam def} respectively. 

In this section, we shall establish the following mapping estimate for the operator $T_q$.
\begin{proposition}[Operator $L^2$ estimate]
\label{prop op mapping}
Let $f$ be a non-negative Schwartz function satisfying
\begin{equation}
    \label{eq fin energy assump}
    \| \widehat{f}(\cdot)\times  <\cdot>^{-\frac{r(\al)}{2}}) \|_2\lesssim 1.
\end{equation} 
Then 
\begin{equation*}
    %\label{eq op L2 est}
    \|(\widehat{T_qf}(\cdot)\times  <\cdot>^{\frac{r(\al)}{2}})\|_2\lesssim 1.
\end{equation*}
\end{proposition}
It follows from Proposition \ref{energy} 
%combined with \eqref{energy_2} 
that $\mu_{q, \tau}$ satisfies the estimate \eqref{eq fin energy assump}. We can thus apply Proposition \ref{prop op mapping}, with $f=\mu_{q, \tau}$, to obtain the desired bound \eqref{Tbound}. We now focus on proving Proposition \ref{prop op mapping}.

Recall the Monge-Ampere determinant of a function $\phi: \mathbb{R}^D\to \mathbb{R}^D\to \mathbb{R}$ from \eqref{eq MongeAmp}, which is the determinant of the matrix
\begin{equation}
    \label{def mat MongeAmp}
    M (\phi)(x, y):=\begin{pmatrix}
         0& \left(\nabla_y\phi\right)^{\intercal} (x, y)\\
        \nabla_x\phi (x, y)& \frac{\partial^2\phi}{\partial_x\partial_y} (x, y)
    \end{pmatrix}.
\end{equation}
The $L^2$ Sobolev mapping property of the operator $T_q$  is determined by the rank of the matrix $M(\phi_{\alpha})$. Let $B$ denote the unit ball in $\mathbb{R}^D$ with respect to the Euclidean norm. For the calculations to follow, it will be easier to work with the function $\Phi: \mathbb{R}^D\times \mathbb{R}^D$ given by
$$\Phi_{\alpha}=\Phi:=\phi^{\alpha}_{\alpha}.$$
Observe that there exist positive constants $c_{\alpha}', c_{\alpha}''$ such that for $(x, y)\in B\times B$ we have
$$c_{\alpha}'|\Phi(x, y)-1|\leq |\phi_{\alpha}(x, y)-1|\leq c_{\alpha}''|\Phi(x, y)-1|.$$
Accordingly, for a smooth compactly supported bump-function $\psi$ supported in a ball centered at the origin in $\mathbb{R}^{D}\times \mathbb{R}^{D}$, we define
\begin{equation}\label{def big op} \mathcal{T}_q f(x)= q^{ \tau }\int_{\{c_{\alpha}'|\Phi_\al(x,y)-1| \le\, q^{- \tau } \}}f(y)\psi(x,y) dy.\end{equation}
Then for any non-negative function $f$,
$$T_q f(x)\leq \mathcal{T}_q f (x).$$

Therefore, to establish Proposition \ref{prop op mapping}, it suffices to show that
\begin{equation}
    \label{eq L2 est}
    \|(\widehat{\mathcal{T}_qf}(\cdot)\times  <\cdot>^{\frac{r(\al)}{2}})\|_2\lesssim 1
\end{equation}
for any non-negative Schwartz function $f$ satisfying \eqref{eq fin energy assump}.

The $L^2$ Sobolev mapping properties of the operator $\mathcal{T}_q$ are determined by the rank of the Monge-Ampere matrix $M(\Phi)$.
In Section \ref{sec rank calc}, we shall establish the following.
\begin{proposition}[Rank Proposition]
\label{prop rank}
Let $\alpha$ be an even positive integer, and let $B$ denote the unit ball in $\mathbb{R}^D$ with respect to the Euclidean norm. For $t>0$, let $(x, y)\in B\times B$ with $\phi_{\alpha}(x,y)=t.$ Finally, set $\Phi=\phi_{\alpha}^{\alpha}$, and let $M(\Phi)(x,y)$ be the matrix in \eqref{def mat MongeAmp}.

(i) We have 
\begin{equation*}
    \label{eq grad est}
    \nabla_{x} \Phi(x, y)\neq 0 \qquad \textrm{ and } \qquad \nabla_{y} \Phi(x, y)\neq 0.
\end{equation*}

(ii) For $\alpha=2$, we have 
\begin{equation*}
    \label{eq MA 4}
    \det M (\Phi)(x, y)=\det \begin{pmatrix}
         0& \left(\nabla_y\Phi\right)^{\intercal} (x, y)\\
        \nabla_x\Phi (x, y)& \frac{\partial^2\Phi}{\partial_x\partial_y}(x, y) 
    \end{pmatrix}
\neq 0.
\end{equation*}

(iii) For $\alpha=4$, we also have 
\begin{equation*}
    %\label{eq MA 4}
    \det M (\Phi)(x, y)
    % =\det \begin{pmatrix}
    %      0& \nabla_x\Phi (x, y)\\
    %     -\left(\nabla_y\Phi\right)^{T} (x, y)& \frac{\partial^2\Phi}{\partial_x\partial_y}(x, y) 
    % \end{pmatrix}
\neq 0.
\end{equation*}

(iv) For $\alpha\geq 6$ and $\oy\neq 0$, we have
\begin{equation*}
    %\label{eq MA not4}
     \det M (\Phi)(x, y)\neq 0.
\end{equation*}
For $\oy=0$, we have
\begin{equation*}
    \label{eq MA not4}
    \mathrm{rank }\, M(\Phi)(x, y) = D=2n+1\,.
\end{equation*}
%Let $\alpha=2$. Then \[\mathrm{rank}\left(\Phi_{(x,y)}''(x,y)\right)=2n.\]
Indeed, when $\oy=0$, 
\begin{equation*}
    \label{eq sub MA not4}
    \det \Tilde{M}(\Phi)(x, y) \neq 0,
\end{equation*}
where $\Tilde{M} (\Phi)$ is a minor of $M(\Phi)$ of order $2n+1$ given by
\begin{equation*}
    \label{def sub MA}
    \widetilde{M} (\Phi)(x, y)=\begin{pmatrix}
         0& \left(\nabla_{\uy}\Phi\right)^{\intercal} (x, y)\\
        \nabla_{\ux}\Phi (x, y)& \frac{\partial^2\Phi}{\partial_\ux\partial_\uy}(x, y) 
    \end{pmatrix}.
\end{equation*}
\end{proposition}

By parts (ii) and (iii) of Proposition \ref{prop rank}, we conclude that for $\alpha=2$ and $\alpha=4$, the Monge-Ampere determinant of of $\Phi=\phi_{\alpha}^{\alpha}$ does not vanish on the set $\{(x, y)\in \mathbb{R}^{D}\times \mathbb{R}^{D}: \phi_{\alpha}(x,y)=t\}$. In other words, the function $\Phi$ satisfies the Phong-Stein rotational curvature condition on this set, and thus
\begin{equation}\label{mapping 4} \mathcal{T}_q: L^2(\mathbb{R}^D) \rightarrow L^2_{\frac{D-1}{2}}(\mathbb{R}^D)\end{equation}
with constants uniform in $t$ and $q$. We refer the reader to \cite{PS86, PhongStein} for the details (see also \cite{Stein}[Chapter XI, \S 3.1] for a background and a thorough description of rotation curvature).
\smallskip

On the other hand, when $\alpha\geq 6$, we conclude from part (iv) of the proposition that either the Monge-Ampere determinant does not vanish (away from the equator when $|\ox-\oy|>0$), or the determinant with respect to the variables $\ux$ and $\uy$ is non-zero (for a fixed choice of the variables $\ox, \oy\in B$). Therefore after a partition of unity and a slicing argument near the equator (when $\ox-\oy=0$), using the Phong-Stein rotational curvature condition when $|\ox-\oy|>0$ and the curvature condition in the restricted choice of variables near the equator, we can conclude that
\begin{equation}\label{mapping not4} \mathcal{T}_q: L^2(\mathbb{R}^D) \rightarrow L^2_{\frac{D-2}{2}}(\mathbb{R}^D).\end{equation}
Recalling the definition of $r(\alpha)$ from \eqref{eq gam def}, we can combine \eqref{mapping 4} and \eqref{mapping not4} to write
\begin{equation}\label{mapping comb} \mathcal{T}_q: L^2(\mathbb{R}^D) \rightarrow L^2_{r(\alpha)}(\mathbb{R}^D).\end{equation}
The Sobolev mapping property \eqref{mapping comb} can be leveraged to establish the estimate \eqref{eq L2 est} for $\mathcal{T}_q$, and consequently Proposition \ref{prop op mapping}, using an argument similar to that used in Section 3 of \cite{IT} (to prove their Proposition 2.2). We briefly sketch the details here.

\begin{proof}[Proof of Proposition \ref{prop op mapping}, given Proposition \ref{prop rank}]
As explained above, it suffices to establish \eqref{eq L2 est} for $\mathcal{T}_q$ satisfying the mapping property \eqref{mapping comb}.

Let $\eta_0(\xi)$ and $\eta(\xi)$ be non-negative Schwartz class functions supported in the ball $\{ |\xi| \leq 4 \}$ and the spherical shell $\{ 1 < |\xi| < 4\}$ respectively, and let $\tilde{\eta}_0(\xi)$ and $\tilde{\eta}(\xi)$ be non-negative Schwartz class functions supported in the slightly bigger ball $\{ |\xi| \leq 4.2^{L} \}$ and  spherical shell $\{ 2^{-L} < |\xi| < 4. 2^{L}\}$ respectively (with $L\geq 1$ is to be determined shortly),
such that
$$\tilde{\eta}_0\eta_0=\eta_0, \qquad \tilde{\eta}\eta=\eta, $$
$$ \eta_j(\xi)=\eta(2^{-j}\xi), \qquad \tilde{\eta}_j(\xi)=\tilde{\eta}(2^{-j}\xi)\qquad \qquad \textrm{ for }j \geq 1,$$ and 
$$\eta_0(\xi) + \sum_{j=1}^{\infty} \eta_j(\xi) =\tilde{\eta}_0(\xi) + \sum_{j=1}^{\infty} \tilde{\eta}_j(\xi) =1.$$ 

For $k \geq0$ and a Schwartz function $g$, let $L_kg$ and $\tilde{L}_kg$ be the Littlewood-Paley projections of $g$ given by
$$ \widehat{L_k(g)}(\xi)=\widehat{g}(\xi) \eta_k(\xi), \qquad \textrm{ and }\qquad \widehat{\tilde{L}_k(g)}(\xi)=\widehat{g}(\xi) \tilde{\eta}_k(\xi).$$

We can now write
\begin{equation}
    \label{eq LP dec}
    \mathcal{T}_qf=\sum_{k\geq 0} \tilde{L}_k\mathcal{T}_q f=\sum_{k\geq 0}\left(\tilde{L}_k\mathcal{T}_q L_kf+E_kf\right),
\end{equation}
where for each $k\geq 0$, the error term $E_kf$ is given by
$$E_k f:=\sum_{j\geq 0:\, j\neq k}\tilde{L}_k\mathcal{T}_q L_jf.$$
Given any $M>0$,  we can choose a large enough $L$ such that by repeated integration by parts, we obtain
%Littlewood-Paley theory 
\begin{equation}
    \label{eq LP err est}
    \|\widehat{E_kf}\times <\cdot>^{\frac{r(\al)}{2}})\|_2\lesssim_M 2^{-kM},
\end{equation}
with the implied constant independent of $k$ (see for example,  \cite[Lemma 3.1]{IT} or \cite[Lemma 2.1]{seeger93}).

For a Schwartz function $g$, we also have the Littlewood-Paley inequality
\begin{equation}
    \label{eq LP ineq}
    \Big\| \sum_{k\ge 0} \tilde{L}_k g\Big \|_2\lesssim \Big\|\Big(\sum_{k\ge 0} | \tilde{L}_kg|^2\Big)^{1/2}\Big  \|_2.
\end{equation}
Using \eqref{eq LP dec}, the triangle inequality and \eqref{eq LP err est}, we can thus estimate
\begin{align*}
    %\label{eq error est}
    &\|(\widehat{\mathcal{T}_qf}\times <\cdot>^{\frac{r(\al)}{2}})\|_2
\\ &\leq \|\sum_{k\geq 0}\widehat{\tilde{L}_k\mathcal{T}_q L_kf}\times <\cdot>^{\frac{r(\al)}{2}})\|_2+\sum_{k\geq 0}\|(\widehat{E_kf}\times <\cdot>^{\frac{r(\al)}{2}})\|_2 
\\ & \leq \|\sum_{k\geq 0}\widehat{\tilde{L}_k\mathcal{T}_q L_kf}\times <\cdot>^{\frac{r(\al)}{2}})\|_2+1.
\end{align*}
To bound the first term above, note that due to \eqref{eq fin energy assump}, we have
\begin{equation}
\label{eq energy LP}
    \sum_{k\geq 0}2^{-k{r(\alpha)}}\|\widehat{L_kf}\|_2^2\sim \|\hat{f}\times<\cdot>^{-\frac{r(\al)}{2}}) \|_2^2\lesssim 1. 
\end{equation}

We can thus estimate
\begin{align*}
    &\|\sum_{k\geq 0}\widehat{\tilde{L}_k\mathcal{T}_q L_kf}\times <\cdot>^{\frac{r(\al)}{2}})\|_2\\
    &\lesssim \left(\sum_{k\geq 0}2^{kr(\alpha)}\int |\widehat{\mathcal{T}_qL_kf}(\xi)|^2\, d\xi\right)^{1/2} \qquad(\text{using }\eqref{eq LP ineq})
    \\&\lesssim \left(\sum_{k\geq 0}2^{kr(\alpha)}2^{-2kr(\alpha)}\int |\widehat{L_kf}(\xi)|^2\, d\xi\right)^{1/2} (\text{using the mapping property of } \mathcal{T}_q\,  \eqref{mapping comb})
    \\&\leq\left(\sum_{k\geq 0}2^{-kr(\alpha)}\int |\widehat{L_kf}(\xi)|^2\, d\xi\right)^{1/2} 
    \\&\lesssim 1,\qquad\text{using }\eqref{eq energy LP}.
\end{align*}
This finishes the proof.
\end{proof}

\section{Proof of the Rank Proposition \ref{prop rank}}
\label{sec rank calc}
In this section, we shall use results about matrix inverses and determinants from \cite{sri24}[Section 6] to determinant the exact rank of the matrix $M(\Phi)$ for different values of $\alpha$, and in particular, in the Kor\'anyi sphere case when $\alpha=4$. 
Given an even  integer $\alpha\geq 2$, we define the function $\Psi: \mathbb{R}^D\to \mathbb{R}$ to be
\begin{equation}
    \label{eq psi def}
    \Psi(z):=\|z\|_{\alpha}^{\alpha}=|\uz|^\alpha+C_{\alpha}|\oz|^{\alpha/2}.
\end{equation}
For ease of notation, we shall use $A$ to denote the constant $C_\alpha$ chosen so that the triangle inequality with respect to the Heisenberg group law is satisfied for the norm $\|\cdot\|_{\alpha}$.
% \[\|(\ux-\uy, \ox-\oy+\frac{1}{2}\ux^\intercal J\uy)\|_\alpha\leq \|(\ux,\ox)\|_{\alpha}+\|(\uy,\oy)\|_{\alpha}.\]
The function $\Phi: \mathbb{R}^D\times \mathbb{R}^D\to \mathbb{R}$ can thus be expressed as
\begin{equation}
    \label{eqn: phase}
    \Phi(x,y)
    %=\|(\ux-\uy, \ox-\oy+\frac{1}{2}\ux^\intercal J\uy)\|_\alpha^\alpha
    =|\ux-\uy|^{\alpha}+A|\ox-\oy+\frac{1}{2}\ux^\intercal J\uy|^{\alpha/2}=\Psi (\Theta(x, y)),
\end{equation}
where we let $\Theta:\mathbb{R}^{D}\times\mathbb{R}^{D}\to \mathbb{R}^{D}$ denote the Heisenberg group translation map
\begin{equation*}
    \label{eq H law}
    \Theta(x,y):=
    \left(\ux-\uy,\ox-\oy+\tfrac{1}{2}\ux^{\intercal}J\uy\right).
\end{equation*}
In the calculation to follow, we shall use the notation 
\begin{equation}
    \label{def mix Hess xy}
    \Phi_{xy}^{''}(x, y):= 
    \frac{\partial^2 \Phi}{\partial x \partial y}(x, y)=\begin{pmatrix}
    \frac{\partial^2 \Phi}{\partial x_i \partial y_j}(x, y)
\end{pmatrix}_{1\leq i, j\leq 2n+1}
\end{equation} for the mixed Hessian of $\Phi$.
For $\alpha\neq 4$, the $2n$-dimensional minor
\begin{equation}
    \label{def mix Hess minor}
    \Phi_{\ux\uy}^{''}(x, y):= 
    \frac{\partial^2 \Phi}{\partial \ux \partial \uy}(x, y)=\begin{pmatrix}
    \frac{\partial^2 \Phi}{\partial x_i \partial y_j}(x, y)
\end{pmatrix}_{1\leq i, j\leq 2n}
\end{equation}
shall play a vital role.

%\subsection{Translation invariance}
\subsection{Initial Simplification}
\label{section trans invar}

The following lemma simplifies the calculation for the Monge-Ampere matrix of $\Phi$, expressing it in terms of the gradient and Hessian of $\Psi$. All vectors should be interpreted as column matrices.

\begin{lemma}
\label{lem trans invar}
 Let $\Phi, \Psi$ and $\Theta$ be as defined above. Then
%  \begin{align}
%      \label{eq MA phi psi}
%      M(\Phi)(x, y)
%      &=\begin{pmatrix}
%      \end{pmatrix}
% \end{align}
\[\mathrm{rank}\, M(\Phi)(x, y)=\mathrm{rank}\,N(\Psi)( \Theta(x,y)),\]
%\[\mathrm{rank}(\Phi''_{x,y}(x,y))=\mathrm{rank}(\Phi''_{x,y}(0, \Theta(x,y)).\]
where
\begin{equation}
    \label{eq def N psi}
    N(\Psi)(z):=\begin{pmatrix}
0 & & \left(\nabla_{2n}\Psi(z)\right)^{\intercal} & & \Psi'_{D}(z)\\
\nabla_{2n}\Psi(z) & &D_{2n}^2\Psi(z)+\tfrac{1}{2}\Psi_{D}'(z)J& &\nabla_{2n}\Psi'_{D}(z)\\
\Psi'_{D}(z) & & \nabla_{2n}\Psi'_{D}(z)& &\Psi''_{DD}(z)
\end{pmatrix}.
\end{equation}
\end{lemma}
\begin{proof}
Let $\Theta_{x}'$ denote the gradient of $\Theta$ with respect to the $x$ variable. Differentiating \eqref{eqn: phase} and using the chain rule, we have 
\begin{equation}
\label{eqn grad calc}
\nabla_{x}\Phi(x, y)= \Theta_{x}'\cdot\nabla\Psi\big|_{\Theta(x, y)}=
%\begin{pmatrix}
% \Theta_{x}'& & 0\\
% 0& & 1
% \end{pmatrix}=\nabla 
\begin{pmatrix}
I_{2n}& & \tfrac{1}{2}J\uy\\
0& & 1
% \\
% 0 & & 0 & & 1
\end{pmatrix}\cdot \nabla\Psi(\Theta(x, y)),
\end{equation}
%Let $D=2n+1$. 
where $J$ is as in \eqref{grouplaw}. Set $z=\Phi (x, y)$. Expanding the matrix product on the right gives
\[\nabla_{x}\Phi(x, y)=\begin{pmatrix}
\nabla_{2n}\Psi(z)+\tfrac{1}{2}\Psi'_{D}(z)J\uy\\ 
\Psi'_{D}(z)
\end{pmatrix},
%{\big|_{\Theta(x, y)}}
%\left(\nabla_{2n}\Psi+\tfrac{1}{2}\Psi'_{D}J\uy,\Psi'_{D}\right)\big|_{\Theta(x, y)}.
\]
Here $\nabla_{2n}\Psi$ denotes the partial derivative of $\Psi$ with respect to the first $2n$ coordinates and $\Psi'_{D}$
%, \Psi'_{t}$ 
denotes its derivatives with respect to the $D=2n+1$-th coordinate.

Next, we take derivatives with respect to the $y$ variable. Another application of chain rule, recalling that $J^{\intercal}=-J$, yields

\begin{align}
&\Phi''_{x,y}(x, y)=\nonumber\\
&\begin{pmatrix}
D_{2n}^2\Psi(z)+\tfrac{1}{2}J\uy\left(\nabla_{2n}\Psi'_{D}(z)\right)^{\intercal}+\tfrac{1}{2}\Psi_{D}'(z)J& &\nabla_{2n}\Psi'_{D}(z)+\tfrac{1}{2}\Psi''_{DD}(z)J\uy\\
\nabla_{2n}\Psi'_{D}(z)& &\Psi''_{DD}(z)
\end{pmatrix}\begin{pmatrix}
-I_{2n}& & 0\\
-\tfrac{1}{2}(J\ux)^{\intercal}& & -1
\end{pmatrix}\nonumber\\
&=\begin{pmatrix}
I_{2n}& & \tfrac{1}{2}J\uy\\
0& & 1\end{pmatrix}\begin{pmatrix}
D_{2n}^2\Psi(z)+\tfrac{1}{2}\Psi_{D}'(z)J& &\nabla_{2n}\Psi'_{D}(z)\\
\nabla_{2n}\Psi'_{D}(z)& &\Psi''_{DD}(z)
\end{pmatrix}\begin{pmatrix}
-I_{2n}& & 0\\
-\tfrac{1}{2}(J\ux)^{\intercal}& & -1
\end{pmatrix}.
\label{eq matrix commut}
\end{align}
Here $D_{2n}^2$ denotes the Hessian of $\Psi$ with respect to the first $2n$ coordinates. 
We also have
\[\nabla_{x}\Phi(x, y)=\begin{pmatrix}
\nabla_{2n}\Psi(z)+\tfrac{1}{2}\Psi'_{D}(z)J\uy\\ 
\Psi'_{D}(z)
\end{pmatrix},
\]
and
\[\nabla_{y}\Phi(x, y)=\begin{pmatrix}
-\nabla_{2n}\Psi(z)-\tfrac{1}{2}\Psi'_{D}(z)J\ux\\ 
-\Psi'_{D}(z)
\end{pmatrix}.
\]
The above equations can be combined with \eqref{eq matrix commut} and expressed together in the matrix form as
\begin{align}
\label{eq matrix conj MA}
    &\begin{pmatrix}
     0& \left(\nabla_y\Phi (x, y)\right)^{\intercal}\\
        \nabla_x\Phi (x, y)& \frac{\partial^2\Phi}{\partial_x\partial_y}(x, y)   
    \end{pmatrix}=\nonumber \\
&\begin{pmatrix}
1& & 0& & 0\\
0 & & I_{2n}& & \tfrac{1}{2}J\uy\\
0& &0& & 1\end{pmatrix}
N(\Psi)(z)\begin{pmatrix}
1& & 0& & 0\\
0& &-I_{2n}& & 0\\
0& &-\tfrac{1}{2}(J\ux)^{\intercal}& & -1
\end{pmatrix},
\end{align}
where
$$N(\Psi)(z):=\begin{pmatrix}
0 & & \left(\nabla_{2n}\Psi(z)\right)^{\intercal} & & \Psi'_{D}(z)\\
\nabla_{2n}\Psi(z) & &D_{2n}^2\Psi(z)+\tfrac{1}{2}\Psi_{D}'(z)J& &\nabla_{2n}\Psi'_{D}(z)\\
\Psi'_{D}(z) & & \nabla_{2n}\Psi'_{D}(z)& &\Psi''_{DD}(z)
\end{pmatrix}.$$

Since the first and last matrices in the product in \eqref{eq matrix conj MA} have non-zero determinant, it follows that the rank of $\Phi''_{x,y}(x, y)$ is the same as the rank of 
$N(\Psi)(z)$.

This finishes the proof. 
\end{proof}
\subsection{Calculating the Derivatives}
We now focus on the matrix $N(\Psi)(y)$, for $y\in \mathbb{R}^D$ with $\|y\|_{\alpha}=t$. To avoid carrying the $\textrm{sign}$ notation everywhere, we shall assume without loss of generality that $\oy\geq 0$. Recall $\Psi$ from \eqref{eq psi def}. Its first derivatives are given by
\begin{equation}
    \label{eq psi first deriv}
    \nabla_{2n}\Psi(y)= \alpha|\uy|^{\alpha-2}\uy, \qquad \Psi'_D(y)=\frac{A\alpha}{2}\oy^{\frac{\alpha-2}{2}}.
\end{equation}

Calculating the second derivatives, for $\alpha\geq 4$, we obtain that $N(\Psi)(y)$ is equal to
\begin{equation}
    \label{eq Npsi alpha4}
    \begin{pmatrix}
0 & & \alpha|\uy|^{\alpha-2}\uy^{\intercal} & & \frac{A\alpha}{2}\oy^{\frac{\alpha-2}{2}}\\
\alpha|\uy|^{\alpha-2}\uy & &\mathfrak{P}(y) & & 0_{2n}\\
\frac{A\alpha}{2}\oy^{\frac{\alpha-2}{2}} & & 0_{2n}^{\intercal}& &\frac{A\alpha(\alpha-2)}{4}(\oy)^{\frac{\alpha-4}{2}}
\end{pmatrix},
\end{equation}
where 
\begin{equation}
    \label{eq def matrix P}
    \mathfrak{P}(y):=\alpha(\alpha-2)|\uy|^{\alpha-4}\uy^\intercal\uy+\alpha|\uy|^{\alpha-2}I_{2n}+\frac{A\alpha}{4}(\oy)^{\frac{\alpha-2}{2}}J.
\end{equation}
The case $\alpha=2$ is even simpler and we have
\begin{equation}
\label{eq Npsi al2}
    N(\Psi)(y)=\begin{pmatrix}
0 & & 2\uy^{\intercal} & & A\\
2\uy & & 2I_{2n}+\frac{A}{2}J& & 0_{2n}\\
A & & 0_{2n}^{\intercal}& &0
\end{pmatrix}.
\end{equation}
It is clear that for $\alpha>4$, $N(\Psi)$ can have rank at most $2n$ at the equator (when $\oy=0$). For $\alpha=2$ or $\alpha=4$, the matrix $N(\Psi)(y)$ can attain full rank. The exact rank depends on the invertibility of the matrix $\mathfrak{P}(y)$
for $\alpha\geq 4$,
and of the matrix 
\begin{equation}
    \label{eq: matrix alpha 2}
    2I_{2n}+\frac{A}{2}J
\end{equation}
when $\alpha=2$. 
\subsection{Invertibility of Matrices}
Recall that $J$ is the standard symplectic matrix of order $2n$ given by
\[J=\begin{pmatrix}0&I_{n}\\-I_n&0\end{pmatrix}.\]
In particular, $J$ is skew-symmetric, $J^2=-I_{2n}$ and $\|J\|=1$ (the norm here denotes the matrix norm).

Next we have an identity for the inverses of matrices of the form $J+g''(w)$ where $J$ is a skew-symmetric matrix of even dimension $d$ with $J^2=-I_{d}$ and $w\in \mathbb{R}^{d}$. A more general version of the following lemma, and the two subsequent corollaries, were proven by the first author in \cite{sri24}.

\begin{lemma}[\cite{sri24}, Lemma 6.2]
\label{lem matrix inv}
Let $J$ be a skew-symmetric matrix with $J^2=-I_{d}$ and let $w\in\mathbb{R}^{d}$. For any $\sigma,\lambda, \kappa\in \mathbb{R}$, we have
\begin{multline}
    \label{eq:invform}
    (\sigma I_{d}+\lambda ww^{\intercal}+\kappa J)^{-1}\\=\frac{1}{\sigma^2+\kappa^2}\bigg[\sigma I_{d}-\kappa J+\frac{(-\sigma\lambda)}{\sigma^2+\kappa^2+\sigma\lambda|w|^2}\left(\sigma ww^{\intercal}-\kappa(Jw)w^{\intercal}\right)\\
    +\frac{\lambda\kappa}{\sigma^2+\kappa^2+\sigma\lambda|w|^2}\left(-\sigma w(Jw)^{\intercal}+\kappa(Jw)(Jw)^{\intercal}\right)\bigg].
\end{multline}

\end{lemma}
\begin{proof}
This can be verified directly, using the facts that $J^2=-I_{d}$, $ww^\intercal w w^\intercal=|w|^2ww^\intercal$ and $w^\intercal J w=0$ (due to $J$ being skew-symmetric).
\end{proof}
The next corollary follow almost immediately.
\begin{corollary}[\cite{sri24}, Corollary 6.3]
\label{cor matrix inv est}
Suppose there exist positive constants $c, C$ such that \[c<\min\,\{\sigma^2+\kappa^2, \sigma^2+\kappa^2+\sigma\lambda|w|^2\}\] and \[\max\{|\alpha|, |\lambda|, |\kappa|, |w|\}<C.\] Then 
%it is clear from \eqref{eq:invform} that 
the entries of the inverse of $\sigma I_{d}+\lambda ww^{\intercal}+\kappa J$ are bounded above by an absolute constant depending on $C$ and $c$. We can thus conclude that 
\[\|(\sigma I_{d}+\lambda ww^{\intercal}+\kappa J)^{-1}\|\lesssim_{c,C,d} 1,\]
where the norm above denotes the spectral norm of the matrix.
We also have that
\begin{equation*}
\label{eq inv matrix est}
    |\det\left(\sigma I_{d}+\lambda ww^{\intercal}+\kappa J\right)^{-1}|\lesssim_{c, C, d} 1.
\end{equation*}
and thus
\begin{equation*}
    |\det\left(\sigma I_{d}+\lambda ww^{\intercal}+\kappa J\right)|\gtrsim_{c, C, d} 1.
\end{equation*}
\end{corollary}

Observe that $(Jw)^{\intercal}w=0$ as $J$ is skew-symmetric. Thus, we can conclude the following using \eqref{eq:invform}.
\begin{corollary}[\cite{sri24}, Corollary 6.4]
\label{cor inn prod1}
We have
\begin{align*}
&(\sigma I_{d}+\lambda ww^{\intercal}+\kappa J)^{-1}w\\
&=\frac{1}{\sigma^2+\kappa^2}\bigg[(\sigma I_{d}-\kappa J)w+\frac{(-\sigma\lambda)}{\sigma^2+\sigma\lambda|w|^2+\kappa^2} (\sigma ww^{\intercal}- \kappa(Jw)w^{\intercal})w \bigg]\\
&=\frac{1}{\sigma^2+\sigma\lambda|w|^2+\kappa^2}(\sigma I_{d}-\kappa J)w.
\end{align*}
\end{corollary}

We now have all ingredients in place for:
\subsection{Concluding the Proof of the Rank Proposition \ref{prop rank}}
We shall assume without loss of generality that $\oy\geq 0$. We first show part (i). Using the chain rule for \eqref{eqn grad calc}, and a corresponding one for the gradient with respect to the $y$ variable, it suffices to show that
$$\nabla \Psi (y)\neq 0_{2n+1}$$
for $y\in \mathbb{R}^D$ with $\|y\|_{\alpha}=t$.
From \eqref{eq psi first deriv}, we have 
\begin{equation*}
    \nabla \Psi (y)=\begin{pmatrix}
     \alpha|\uy|^{\alpha-2}\uy\\ 
     \frac{A\alpha}{2}\oy^{\frac{\alpha-2}{2}}
    \end{pmatrix},
\end{equation*}
which is not the zero vector since $\|y\|_{\alpha}=t\neq 0$. This establishes part (i).

Next, we determine the rank of $M(\Phi)(x, y)$ for $\alpha\geq 2$. Due to Lemma \ref{lem trans invar}, it suffices to determine the rank of $N(\Psi)(y)$ instead, under the assumption that $\|y\|_{\alpha}=t$. We begin with the case $\alpha=2$, in which case, it is clear from \eqref{eq Npsi al2} that
$$\textrm{rank}\, N(\Psi)(y)= 2+\textrm{rank}\, \left(2I_{2n}+J\right).$$
Next, using Corollary \ref{cor matrix inv est} with 
$\sigma= 2$, $\kappa=1$ and $\lambda=0$, we conclude that the matrix $2I_{2n}+J$ is invertible, and therefore
$$\textrm{rank}\, N(\Psi)(y)= 2+2n=D+1.$$
This establishes part (ii).

Next, we deal with the cases when $\alpha=4$, or when $\alpha\geq 6$ and $|\oy|\neq 0$, which are slightly more involved. 

Recall the matrices $N(\Psi)(y)$ and $\mathfrak{P}(y)$ from \eqref{eq Npsi alpha4} and \eqref{eq def matrix P} in this case. First, using Corollary \ref{cor matrix inv est} with 
\begin{equation}
     \label{eq par choice}
     \sigma=\alpha |\uy|^{\alpha-2},\qquad \kappa=\frac{A\alpha}{4}(\oy)^{\frac{\alpha-2}{2}},\qquad \lambda=\alpha(\alpha-2)|\uy|^{\alpha-4},
\end{equation}
we conclude that the matrix $\mathfrak{P}(y)$ is invertible.

We now use the identity
\begin{equation}
    \label{eq mat id}
    \det \begin{pmatrix}
    \mathfrak{A} & \mathfrak{B}\\
    \mathfrak{C} & \mathfrak{D}
\end{pmatrix}= \det \left( \mathfrak{A}- \mathfrak{B}\mathfrak{D}^{-1}  \mathfrak{C}\right) \det \mathfrak{D},
\end{equation}
for a square matrix expressed in the block form, where $\mathfrak{A}$ and $\mathfrak{D}$ are also square matrices with $\mathfrak{D}$ being invertible.
Applying this for the matrix $N(\Psi)(y)$, 
we conclude that
$$\det N(\Psi)(y)= \mathfrak{X}(y)\,  \det \begin{pmatrix}
    \mathfrak{P}(y) & & 0_{2n}\\
    0_{2n}^{\intercal} & & \frac{A\alpha(\alpha-2)}{4}(\oy)^{\frac{\alpha-4}{2}}
\end{pmatrix},$$
where
\begin{equation}
    \label{eq def X}
    \mathfrak{X}(y):= \begin{pmatrix}
    \alpha|\uy|^{\alpha-2}\uy\\
    %+\frac{A\alpha}{4}\oy^{\frac{\alpha-2}{2}}J\uy\\   
    \frac{A\alpha}{2}\oy^{\frac{\alpha-2}{2}}
    \end{pmatrix}^{\intercal}\, \begin{pmatrix}
    \mathfrak{P}(y)^{-1} & & 0_{2n}\\
    0_{2n}^{\intercal} & & \left(\frac{A\alpha(\alpha-2)}{4}(\oy)^{\frac{\alpha-4}{2}}\right)^{-1}
\end{pmatrix}\, \begin{pmatrix}
    \alpha|\uy|^{\alpha-2}\uy\\
    %+\frac{A\alpha}{4}\oy^{\frac{\alpha-2}{2}}J \uy\\
    \frac{A\alpha}{2}\oy^{\frac{\alpha-2}{2}}
\end{pmatrix}
\end{equation}

Thus in order to establish the invertiblilty of the matrix $N(\Psi)(y)$, it is enough to show that
\begin{equation}
    \label{eq X nzero}
    \mathfrak{X}(y)\neq 0.
\end{equation}
Using Corollary \ref{cor inn prod1} 
%and \ref{cor inn prod2} 
with the choice of parameters in \eqref{eq par choice}, we get
\begin{align*}
    \left(\sigma \uy\right)^{\intercal}\mathfrak{P}(y)^{-1}\left(\sigma \uy
    %+\kappa J\uy
    \right)= \frac{\sigma^2\uy^{\intercal}(\sigma I_{2n}-\kappa J)\uy}{\sigma^2+\sigma\lambda|\uy|^2+\kappa^2}
    % \left[\sigma(\sigma I_{2n}-\kappa J)\uy+\kappa\left(\kappa I_{2n}+(\sigma+\lambda|\uy|^2) J\right)J\uy\right]\\
    = \frac{\sigma^3|\uy|^2}{\sigma^2+\sigma\lambda|\uy|^2+\kappa^2}= \frac{\alpha^3|\uy|^{3\alpha-4}}{\sigma^2+\sigma\lambda|\uy|^2+\kappa^2}
    % &=\frac{1}{\sigma^2+\sigma\lambda|\uy|^2+\kappa^2}\left[(\sigma^2-\kappa\sigma-\kappa\lambda|\uy|^2)I_{2n}-(\sigma \kappa+\kappa^2)J\right]\uy,
\end{align*}
and therefore
\begin{align*}
    \mathfrak{X}(y)=\frac{\alpha^3}{\sigma^2+\sigma\lambda|\uy|^2+\kappa^2}\, |\uy|^{3\alpha-4}+ \frac{A\alpha}{\alpha-2}\oy^{\frac{\alpha}{2}}>0,
\end{align*}
whenever $|\uy|^{\alpha}+A(\oy)^{\frac{\alpha}{2}}=t^{\alpha}>0$.

This establishes part (iii) (the case when $\alpha=4$) and also part (iv) for $\alpha\geq 6$ in the case when $\oy>0$.

Finally, we consider the case when $\alpha\geq 6$ and $\oy=0$. In this regime, the last row and columns of the matrix $N(\Psi)(y)$ from \eqref{eq Npsi alpha4} are zero, and its rank is equal to that of the $D$ dimensional submatrix
\begin{equation}
    \label{eq subNpsi alpha4}
    \widetilde{N}(\Psi)(\uy):=\begin{pmatrix}
0 & & \alpha|\uy|^{\alpha-2}\uy^{\intercal} \\
\alpha|\uy|^{\alpha-2}\uy & &\mathfrak{R}(\uy)
\end{pmatrix},
\end{equation}
where 
\begin{equation}
    \label{eq def matrix sub P}
    \mathfrak{R}(\uy):=\alpha(\alpha-2)|\uy|^{\alpha-4}\uy^\intercal\uy+\alpha|\uy|^{\alpha-2}I_{2n}.
\end{equation}
Note that $|\uy|\neq 0$ in this subregime. Using Corollary \ref{cor matrix inv est} again, with 
\begin{equation}
     \label{eq par choice 2}
     \sigma=\alpha |\uy|^{\alpha-2},\qquad \kappa=0,\qquad \lambda=\alpha(\alpha-2)|\uy|^{\alpha-4},
\end{equation}
we conclude that the matrix $\mathfrak{R}(\uy)$ is invertible. Next, we apply the identity \eqref{eq mat id} for the matrix $\widetilde{N}(\Psi)(\uy)$ to get
$$\det \widetilde{N}(\Psi)(y)= \mathfrak{Z}(\uy)\,  \det \mathfrak{R}(\uy)$$
where
\begin{equation}
    \label{eq def Z}
    \mathfrak{Z}(\uy):= \left(\alpha|\uy|^{\alpha-2}\uy\right)^{\intercal} \mathfrak{R}(\uy)^{-1}\left(\alpha|\uy|^{\alpha-2}\uy\right).
\end{equation}
Thus we will be done if show that $ \mathfrak{Z}(y)\neq 0$. The proof is exactly the same (in fact, simpler) as that of \eqref{eq X nzero}. Using Corollary \ref{cor inn prod1} 
%and \ref{cor inn prod2} 
with the choice of parameters in \eqref{eq par choice 2}, we get
\begin{align*}
    \mathfrak{Z}(\uy)=\left(\sigma \uy\right)^{\intercal}\mathfrak{R}(y)^{-1}\left(\sigma \uy
    %+\kappa J\uy
    \right)= \frac{\alpha^3|\uy|^{3\alpha-4}}{\sigma^2+\sigma\lambda|\uy|^2+\kappa^2}\neq 0,
\end{align*}
since $|\uy|^{\alpha}=t^{\alpha}>0$. Thus it follows that the matrix $\widetilde{N}(\Psi)(\uy)$ has rank $D=2n+1$. Consequently, the same holds true for ${N}(\Psi)(y)$ when $\oy=0$. This establishes part (iv), and finishes the proof.

%\bibliography{ref}

\bibliographystyle{alpha}

\end{document}